\documentclass{amsart}
\usepackage[english]{babel}
\usepackage{amssymb,enumerate,bbm,amsmath}
\numberwithin{equation}{section}
\newtheorem{theorem}{Theorem}[section]

\newtheorem*{mthm}{Main Theorem}
\newtheorem{corollary}[theorem]{Corollary}
\newtheorem{lemma}[theorem]{Lemma}
\newtheorem{proposition}[theorem]{Proposition}
\newtheorem*{example}{Example}
\newtheorem*{conji}{Conjecture A}
\newtheorem*{conjii}{Conjecture B}

\newcommand{\mathd}{\mathrm{d}}
\newcommand{\mathe}{\mathrm{e}}
\newcommand{\tmdummy}{$\mbox{}$}
\newcommand{\tmop}[1]{\ensuremath{\operatorname{#1}}}
\newcommand{\tmscript}[1]{\text{\scriptsize{$#1$}}}
\newenvironment{enumerateroman}{\begin{enumerate}[\textup{(}i\textup{)}] }{\end{enumerate}}

\newcommand{\dt}{\ensuremath{\frac{\partial}{\partial t}}}
\newcommand{\normho}{\ensuremath{| \mathring{h} |}}

\begin{document}

\title[An Optimal Convergence Theorem for Mean Curvature Flow]
{An Optimal Convergence Theorem for Mean Curvature Flow of Arbitrary
Codimension in Hyperbolic Spaces}

\author{Li Lei and Hongwei Xu}
\address{Center of Mathematical Sciences \\ Zhejiang University \\ Hangzhou 310027 \\ China\\}
\email{lei-li@zju.edu.cn; xuhw@cms.zju.edu.cn}
\date{}
\thanks{Research supported by the National Natural Science Foundation of China, Grant No. 11371315.}

\begin{abstract}
  In this paper, we prove that if the initial submanifold $M_0$ of
dimension $n(\ge6)$ satisfies an optimal pinching condition, then
the mean curvature flow
  of arbitrary codimension in hyperbolic spaces converges to a round
  point in finite time. In particular, we obtain the optimal
differentiable sphere theorem for submanifolds in hyperbolic spaces.
It should be emphasized that our pinching condition implies that the
Ricci curvature of the initial submanifold is positive, but does not
imply positivity of the sectional curvature of $M_0$.
\end{abstract}

{\maketitle}

\section{Introduction}

The investigation of curvature and topology of manifolds is one of
the main stream in global differential geometry. The sphere theorem
for compact manifolds was initiated by Rauch in 1951. During the
past six decades, there are many important progresses on sphere
theorems for Riemannian manifolds and submanifolds
\cite{Berger,Brendle2,MR2449060,GS,Hamilton,Shiohama}. As is known,
the theory of curvature flows become more and more important in the
geometry and topology of manifolds
\cite{MR2739807,BW,Brendle1,Brendle2,MR2449060,BS2,MR3005061,Hamilton,Huisken2,MR2483374,XG2,Zhu},
etc. In \cite{MR2449060}, Brendle and Schoen proved the remarkable
differentiable 1/4-pinching sphere theorem via the Ricci flow, which
had been open for half a century. Since the dimension of a complex
projective space is always even, Brendle and Schoen's differentiable
sphere theorem is optimal for even dimensional cases. In \cite{BS2},
Brendle and Schoen obtained a differentiable rigidity theorem for
compact manifolds with weakly 1/4-pinched curvatures in the
pointwise sense.

Let $M^{n}$ be an $n(\ge2)$-dimensional submanifold in
an $(n+q)$-dimensional simply connected space form $\mathbb{F}^{n+q}
( c )$ with constant curvature $c$. Denote by $H$ and $h$ the mean
curvature vector and the second fundamental form of $M$,
respectively. Set
\begin{equation}
\alpha ( n, |H| ,c ) =n c+ \frac{n}{2 ( n-1 )} |H|^{2} -
\frac{n-2}{2
   ( n-1 )} \sqrt{|H|^{4} +4 ( n-1 ) c |H|^{2}}.
\end{equation}
After the pioneering rigidity theorem for closed minimal
submanifolds in a sphere due to Simons \cite{Simons}, Lawson
\cite{MR0238229} and Chern-do Carmo-Kobayashi \cite{CDK} obtained a
classification of $n$-dimensional closed minimal submanifolds in
$\mathbb{S}^{n+q}$ whose squared norm of the second fundamental form
satisfies $| h |^{2}\le n/(2-1/q)$. Later Li-Li \cite{MR1161925}
improved Simons' pinching constant for $n$-dimensional closed
minimal submanifolds in $\mathbb{S}^{n+q}$ to
$\mathrm{max}\{\frac{n}{2-1/q},\frac{2}{3}n\}$. Putting $\alpha_1 (
n, | H | ) =\alpha ( n, |H| ,1),$ the second author
\cite{Xu1,MR1241055} proved the generalized Simons-Lawson-Chern-do
Carmo-Kobayashi theorem for compact submanifolds with parallel mean
curvature in $\mathbb{S}^{n+q}$ whose squared norm of the second
fundamental form satisfies $| h |^{2}\le C(n,q,|H|).$ Here
$$C(n,q,|H|)=\left\{\begin{array}{llll} \alpha_1(n,|H|),&\mbox{\ for\ } q=1, \mbox{\ or\ } q=2 \mbox{\ and\ }
 H\neq0,\\
\min\Big\{\alpha_1(n,|H|),\frac{2n}{3}+\frac{5|H|^2}{3n}\Big\},&\mbox{\
otherwise.\ }
\end{array} \right.$$

In the case where $c<0$, the second author {\cite{Xu1,xuhwrigidity}}
proved the following optimal rigidity theorem for submanifolds with
parallel mean curvature in hyperbolic spaces.
\begin{theorem}
  Let $M$ be an n-dimensional ($n \ge 3$) complete submanifold with
  parallel mean curvature in the hyperbolic space $\mathbb{H}^{n+q} ( c )$.
  If $\sup_{M} ( | h |^{2} - \alpha ( n, | H | ,c ) ) <0$, where $| H |^2+n^2c>0$,
  then $M$ is the totally umbilical sphere $\mathbb{S}^{n} ( n/
  \sqrt{| H |^2+n^2c} )$.
\end{theorem}

Using nonexistence for stable currents on compact submanifolds of a
sphere and the generalized Poincar\'{e} conjecture in dimension
$n(\ge5)$ verified by Smale, Lawson and Simons \cite{LS} proved
that if $M^{n}(n\ge 5)$ is an oriented compact submanifold in
$\mathbb{S}^{n+p}$, and if $|h|^2<2\sqrt{n-1}$, then $M$ is homeomorphic to a
sphere. Notice that $\min_{|H|}\alpha(n,|H|,1)=2\sqrt{n-1}.$
Shiohama and Xu \cite{MR1458750} improved Lawson-Simons' result and
proved the optimal sphere theorem.

\begin{theorem}
Let $M$ be an n-dimensional ($n \ge 4$) oriented complete
submanifold in $\mathbb{F}^{n+q} ( c )$ with $c\ge 0$. Suppose that
$\sup_{M} ( | h |^{2} -
  \alpha ( n, | H | ,c ) ) <0.$ Then $M$ is homeomorphic to a
sphere.
\end{theorem}

The following problem is very attractive: \emph{Is it possible to
generalize Theorem 1.2 to the case of submanifolds in hyperbolic
spaces?} By investigating nonexistence for stable currents on
compact submanifolds, Fu and Xu \cite{MR2437072} obtained partial
solution to this problem.

Let $F_{0} :M^{n} \rightarrow N^{n+q} $ be an
$n$-dimensional submanifold smoothly immersed in a Riemannian manifold.
The mean curvature flow with initial value
$F_{0}$ is a smooth family of immersions $F: M \times [ 0,T )
\rightarrow N^{n+q} $ satisfying
\begin{equation}
  \left\{\begin{array}{l}
     \dt F ( x,t ) =H ( x,t ) ,\\
     F ( \cdot ,0 ) =F_{0} ,
   \end{array}\right.
\end{equation}
where $H ( x,t )$ is the mean curvature vector of the submanifold
$M_{t} =F_{t} ( M )$, $F_{t} =F ( \cdot ,t )$.

In 1984, Huisken {\cite{MR772132}} first proved that uniformly
convex hypersurfaces in Euclidean space will converge to a round
point along the mean curvature flow. Further discussions on
convergence results for the mean curvature flow of hypersurfaces in
certain Riemannian manifolds have been carried out by many other
authors {\cite{MR837523,MR892052,lx2014}}, etc. After the important
work on convergence results for the mean curvature flow of arbitrary
codimension in Euclidean spaces and spheres due to Andrews and Baker
{\cite{MR2739807,baker2011mean}}, Liu, Xu, Ye and Zhao
{\cite{MR3078951}} proved the following convergence result for
pinched submanifolds in hyperbolic spaces.

\begin{theorem}
  Let $F_{0} :M^{n} \rightarrow \mathbb{H}^{n+q} ( c )$
  be an n-dimensional ($n \ge 4$) compact submanifold immersed in the hyperbolic
  space. If
  $F_{0}$ satisfies $| h |^{2} \le \frac{1}{n-1} | H |^{2} +2c,$
  then the mean curvature flow with initial value $F_{0}$ has a unique smooth
  solution on a finite maximal time interval $[ 0,T )$, and the solution
  $F_{t}$ converges to a round point as $t \rightarrow T$.
\end{theorem}

When $n\ge4$, we have the unified convergence theorem
{\cite{MR2739807,baker2011mean,MR3078951}} of the mean curvature
flow in space forms under the pinching condition $| h |^{2} \le
\frac{1}{n-1} | H |^{2} +2c$, where $| H |^2+n^2c>0$. The initial
submanifold satisfying the above pinching condition possesses
non-negative sectional curvature. Meanwhile, Gu and Xu
{\cite{MR3005061,XG1}} obtained a convergence theorem for the Ricci
flow of submanifolds in space forms under the same pinching
condition. For any fixed positive constant $\varepsilon$, there are
examples {\cite{MR2739807,baker2011mean,MR892052,MR3078951}} which
show that the pinching condition can not be improved to $| h |^{2} <
\frac{1}{n-1} | H |^{2} +2c+\varepsilon$. Motivated by the rigidity,
sphere and convergence theorems above, i.e., Theorems 1.1-1.3, Liu,
Xu and Zhao {\cite{liu2013mean}} proposed the following.

\begin{conji}
Let $M_{0}$ be an n-dimensional ($n \ge 3$) complete
submanifold immersed in the hyperbolic space $\mathbb{H}^{n+q} ( c
)$. If $M_{0}$ satisfies $\sup_{M_{0}} ( | h |^{2} - \alpha ( n, | H
| ,c ) ) <0$ and $| H |^{2} +n^{2} c>0,$ then the mean curvature
flow with initial value $M_{0}$ has a unique smooth solution on a
finite maximal time interval $[ 0,T )$, and the solution $M_{t}$
converges to a round point as $t \rightarrow T$. In particular,
$M_0$ is diffeomorphic to the standard $n$-sphere
 $\mathbb{S}^{n}$.
\end{conji}

The purpose of the present article is to prove Conjecture A in
dimension $n(\ge6)$. We will establish the optimal convergence
theorem for the mean curvature flow of arbitrary codimension in
hyperbolic spaces, which implies the optimal differentiable sphere
theorem.

\begin{mthm}
Let $F_{0} :M^{n} \rightarrow \mathbb{H}^{n+q} ( c )$ be an n-dimensional ($n
\ge 6$) complete submanifold immersed in the hyperbolic space with
constant curvature $c$. If $F_{0}$ satisfies
\[ \sup_{F_{0}} ( | h |^{2} - \alpha ( n, | H | ,c ) ) <0 \hspace{1em}
   \tmop{and} \hspace{1em} | H |^{2} +n^{2} c>0, \]
then the mean curvature flow with initial value $F_{0}$ has a unique smooth
solution $F: M \times [ 0,T ) \rightarrow \mathbb{H}^{n+q} ( c )$ on a finite
maximal time interval, and $F_t$ converges to a round point as $t \rightarrow
T$. In particular, $M$ is diffeomorphic to the standard $n$-sphere
$\mathbb{S}^{n}$.
\end{mthm}

Since $\alpha ( n, | H | ,c ) > \frac{1}{n-1} | H |^{2} +2c$, our
main theorem improves Theorem 1.3 for $n\ge6$. Note that every
initial submanifold in the convergence results
\cite{andrews2002mean,MR2739807,baker2011mean,MR3005061,MR772132,MR837523,MR892052,MR3078951,liu2012mean,XG1}
possesses quasi-positive curvature. The pinching condition in Main Theorem
implies that the Ricci curvature of the initial submanifold is
positive \cite{MR1458750}, but does not imply positivity of the
sectional curvature. The following example shows the pinching
condition in Main Theorem is optimal for arbitrary $n (\ge 6)$.

\begin{example}
  Let $\lambda$, $\mu$ be positive constants satisfying $\lambda   \mu =-c$
  and $\lambda > \sqrt{-c}$, where $c<0$. For $n \ge 3$, we consider the
  submanifold $M=\mathbb{F}^{n-1} ( c+ \lambda^{2} ) \times \mathbb{F}^{1} (
  c+ \mu^{2} ) \subset \mathbb{H}^{n+q} ( c )$. Then $M$ is a complete
  submanifold with parallel mean curvature, which satisfies $| H | \equiv (
  n-1 ) \lambda + \mu >n \sqrt{-c}$ and $| h |^{2} \equiv ( n-1 ) \lambda^{2} +
  \mu^{2} = \alpha ( n, | H | ,c )$.
\end{example}

The key ingredient of the proof of Main Theorem is to establish the
elaborate estimates for the pinching quantity $\mathring{\alpha} =
\alpha ( n, | H | ,c ) - \frac{1}{n} | H |^{2}$, because our
pinching condition is sharper than that in Theorem 1.3. Using the
properties of $\mathring{\alpha}$ and the evolution equations, we
first derive that $\normho^{2} < \mathring{\alpha}$ is preserved
along the mean curvature flow. Applying a new auxiliary function
$f_{\sigma} = \normho^{2}/\mathring{\alpha}^{1- \sigma}$, we deduce
that $\normho^{2} \le C_{0}   | H |^{2 ( 1- \sigma )}$ via the De
Giorgi iteration. We then obtain an estimate for $| \nabla H |$.
Finally, using estimates for $| \nabla H |$ and the Ricci curvature,
we show that $\tmop{diam} M_{t} \rightarrow 0$ and $| H
  |_{\min} / | H |_{\max} \rightarrow 1$ as $t
\rightarrow T$. This implies the flow shrinks to a round point.

\section{Notations and formulas}

Let $( M^{n} ,g )$ be a Riemannian submanifold immersed in a space form
$\mathbb{F}^{n+q} ( c )$ with constant curvature $c$. We denote by
$\bar{\nabla}$ the Levi-Civita connection of the ambient space
$\mathbb{F}^{n+q} ( c )$. We use the same symbol $\nabla$ to represent the
connections of the tangent bundle $T M$ and the normal bundle $N M$. Denote by
$( \cdot )^{\top}$ and $( \cdot )^{\bot}$ the projections onto $T M$ and $N
M$, respectively. For $u,v \in \Gamma ( T M )$, $\xi \in \Gamma ( N M )$, the
connections $\nabla$ is given by $\nabla_{u} v= ( \bar{\nabla}_{u} v )^{\top}$
and $\nabla_{u} \xi = ( \bar{\nabla}_{u} \xi )^{\bot}$. The second fundamental
form of $M$ is defined as
\[ h ( u,v ) = ( \bar{\nabla}_{u} v )^{\bot} . \]

Let $\{ e_{i}   \,|\,  1 \le i \le n \}$ be a local orthonormal frame
for the tangent bundle and $\{ \nu_{\alpha}   \,|\,  1 \le \alpha \le
q \}$ be a local orthonormal frame for the normal bundle. Let $\{ \omega_{i}
\}$ be the dual frame of $\{ e_{i} \}$. With the local frame, the first and
second fundamental forms can be written as $g= \sum_{i} \omega^{i} \otimes
\omega^{i}$ and $h= \sum_{i,j, \alpha} h^{\alpha}_{i j}   \omega^{i} \otimes
\omega^{j} \otimes \nu_{\alpha}$, respectively. The mean curvature vector is
given by
\[ H= \sum_{\alpha} H^{\alpha} \nu_{\alpha} , \hspace{1em} H^{\alpha} =
   \sum_{i} h^{\alpha}_{i i} . \]
We denote by $\nabla^{2}_{i,j} T= \nabla_{i} ( \nabla_{j} T ) -
\nabla_{\nabla_{i} e_{j}} T$ the second order covariant derivative of tensor.
Then the Laplacian of a tensor is defined by $\Delta T= \sum_{i}
\nabla^{2}_{i,i} T$.

We have the following estimates for the gradient of second fundamental form.

\begin{lemma}
  \label{dA2}For every submanifold in a space form, we have
  \begin{enumerateroman}
    \item $| \nabla h |^{2} \ge \frac{3}{n+2} | \nabla H |^{2}$,

    \item $| \nabla | H |^{2} | \le 2 | H |   | \nabla H |$.
  \end{enumerateroman}
\end{lemma}

The proof of (i) is the same as in {\cite{MR2739807,MR772132}}, and
(ii) follows from the Cauchy-Schwarz inequality.

Let $\mathring{h} =h- \tfrac{1}{n} g \otimes H$ be the traceless
second fundamental form of $M$. Its norm is given by $\normho^{2} =
| h |^{2} - \frac{1}{n} | H |^{2}$. As in
{\cite{MR2739807,baker2011mean}}, we define the following scalars on
$M$.
\[ R_{1} = \sum_{\alpha , \beta} \left( \sum_{i,j} h^{\alpha}_{i j}
   h^{\beta}_{i j} \right)^{2} + \sum_{i,j, \alpha , \beta} \left( \sum_{k} (
   h^{\alpha}_{i k}  h^{\beta}_{j k} -h^{\beta}_{i k}  h^{\alpha}_{j k} )
   \right)^{2} , \]
\[ R_{2} = \sum_{i,j} \left( \sum_{\alpha} H^{\alpha} h^{\alpha}_{i j}
   \right)^{2} , \]
\[ W=n c \normho -R_{1} + \sum_{i,j,k, \alpha , \beta} H^{\alpha} h_{i
   k}^{\alpha} h_{i j}^{\beta} h_{j k}^{\beta} . \]
By a direct computation, we get the following identity for the
Laplacian of $\normho^{2}$.

\begin{equation}
  \frac{1}{2} \Delta \normho^{2} = \left\langle \mathring{h} , \nabla^{2} H
  \right\rangle + | \nabla h |^{2} - \frac{1}{n} | \nabla H |^{2} + W.
  \label{lapho2}
\end{equation}

At a fixed point in $M$, we choose an orthonormal frame $\{
\nu_{\alpha} \}$ for the normal space, such that $H= | H | \nu_{1}$,
and an orthonormal frame $\{ e_{i} \}$ for the tangent space, such
that $( h^{1}_{i j} )$ is diagonal. Then $\left( \mathring{h}^{1}_{i
j} \right)$ is also diagonal, we denote its diagonal elements by
$\mathring{\lambda}_{i}$. Thus $\mathring{\lambda}_{i} =h^{1}_{i i}
- \frac{1}{n} | H |$ and $\mathring{h}^{\alpha}_{i j} =h^{\alpha}_{i
j}$ for $\alpha >1$. We split $\normho^{2}$ into three parts
\begin{equation}
  \normho^{2} =P_{1} +P_{2} , \hspace{1em} P_{2} =Q_{1} +Q_{2} , \label{hopq}
\end{equation}
where
\[ P_{1} = \sum_{i} \mathring{\lambda}_{i}^{2} , \hspace{1em} Q_{1} =
   \sum_{\tmscript{\begin{array}{c}
     \alpha >1\\
     i
   \end{array}}} \left( \mathring{h}^{\alpha}_{i i} \right)^{2} , \hspace{1em}
   Q_{2} = \sum_{\tmscript{\begin{array}{c}
     \alpha >1\\
     i \neq j
   \end{array}}} \left( \mathring{h}^{\alpha}_{i j} \right)^{2} . \]
With the special frame, $R_{1}$ becomes
\begin{eqnarray*}
  R_{1} & = & P_{1}^{2} + \frac{2}{n} P_{1} | H |^{2} + \frac{1}{n^{2}} | H
  |^{4}\\
  &  & +2 \sum_{\alpha >1} \left( \sum_{i} \mathring{\lambda}_{i}
  \mathring{h}^{\alpha}_{i i} \right)^{2} + \sum_{\alpha , \beta >1} \left(
  \sum_{i,j} \mathring{h}^{\alpha}_{i j}   \mathring{h}^{\beta}_{i j}
  \right)^{2}\\
  &  & +2 \sum_{\tmscript{\begin{array}{c}
    \alpha >1\\
    i \neq j
  \end{array}}} \left( \left( \mathring{\lambda}_{i} - \mathring{\lambda}_{j}
  \right) \mathring{h}^{\alpha}_{i j} \right)^{2} +
  \sum_{\tmscript{\begin{array}{c}
    \alpha , \beta >1\\
    i,j
  \end{array}}} \left( \sum_{k} \left( \mathring{h}^{\alpha}_{i k}
  \mathring{h}^{\beta}_{j k} - \mathring{h}^{\alpha}_{j k}
  \mathring{h}^{\beta}_{i k} \right) \right)^{2} .
\end{eqnarray*}
By the Cauchy-Schwarz inequality, we have
\begin{equation}
  \sum_{\alpha >1} \left( \sum_{i} \mathring{\lambda}_{i}
  \mathring{h}^{\alpha}_{i i} \right)^{2} \le \sum_{\alpha >1} \left(
  \sum_{i} \mathring{\lambda}_{i}^{2} \right) \left( \sum_{i} \left(
  \mathring{h}^{\alpha}_{i i} \right)^{2} \right) =P_{1} Q_{1} . \label{P1Q1}
\end{equation}
We also have
\begin{equation}
  \sum_{\tmscript{\begin{array}{c}
    \alpha >1\\
    i \neq j
  \end{array}}} \left( \left( \mathring{\lambda}_{i} - \mathring{\lambda}_{j}
  \right) \mathring{h}^{\alpha}_{i j} \right)^{2} \le
  \sum_{\tmscript{\begin{array}{c}
    \alpha >1\\
    i \neq j
  \end{array}}} 2 \left( \mathring{\lambda}_{i}^{2} +
  \mathring{\lambda}_{j}^{2} \right) \left( \mathring{h}^{\alpha}_{i j}
  \right)^{2} \le 2P_{1} Q_{2} .
\end{equation}
It follows from Theorem 1 of {\cite{MR1161925}} that
\begin{equation}
  \sum_{\alpha , \beta >1} \left( \sum_{i,j} \mathring{h}^{\alpha}_{i j}
  \mathring{h}^{\beta}_{i j} \right)^{2} + \sum_{\tmscript{\begin{array}{c}
    \alpha , \beta >1\\
    i,j
  \end{array}}} \left( \sum_{k} \left( \mathring{h}^{\alpha}_{i k}
  \mathring{h}^{\beta}_{j k} - \mathring{h}^{\alpha}_{j k}
  \mathring{h}^{\beta}_{i k} \right) \right)^{2} \le \frac{3}{2}
  P_{2}^{2} . \label{sP232}
\end{equation}
Then we obtain
\begin{equation}
  R_{1} \le P_{1}^{2} + \frac{2}{n} P_{1} | H |^{2} + \frac{1}{n^{2}} |
  H |^{4} +2P_{1} Q_{1} +4P_{1} Q_{2} + \frac{3}{2}  P_{2}^{2} .
  \label{R1nogreater}
\end{equation}

We also have
\begin{equation}
  R_{2} = \sum_{i,j} ( | H | h^{1}_{i j} )^{2} = | H |^{2} \left( P_{1} +
  \frac{1}{n} | H |^{2} \right) . \label{R2eq}
\end{equation}
Combining (\ref{hopq}), (\ref{R1nogreater}) and (\ref{R2eq}), we obtain

\begin{lemma} For every submanifold in a space form, we have
  \label{R12}{\tmdummy}

  \begin{enumerateroman}
    \item
    \[ R_{1} - \frac{1}{n} R_{2} \le \normho^{4} + \frac{1}{n}
       \normho^{2} | H |^{2} +2P_{2} \normho^{2} - \frac{1}{n} P_{2} | H |^{2}
       , \]
    \item
    \[ R_{2} = \normho^{2} | H |^{2} + \frac{1}{n} | H |^{4} -P_{2} | H |^{2}
       . \]
  \end{enumerateroman}
\end{lemma}

The following proposition (see {\cite{MR1289187,MR1633163}}) will be
used to estimate a cubic polynomial of $h_{i j}^{\alpha}, \, 1\le
i,j\le n,\, 1\le \alpha \le q.$

\begin{proposition}\label{lao3}
  Let $a_{1} , \cdots ,a_{n}$ and $b_{1} , \cdots ,b_{n}$ be real numbers
  satisfying $\sum_{i} a_{i} = \sum_{i} b_{i} =0$. Then
  \[ \left| \sum_{i} a_{i} b_{i}^{2} \right| \leqslant \frac{n-2}{\sqrt{n (
     n-1 )}} \left( \sum_{i} a_{i}^{2} \right)^{\frac{1}{2}} \left( \sum_{i}
     b_{i}^{2} \right) , \]
  where equality holds if and only if $\sum_{i} a_{i}^{2} =0$, or $\sum_{i}
  b_{i}^{2} =0$, or at least $n-1$ pairs of numbers of $( a_{i} ,b_{i} )$ are
  equal.

  In particular, we have
  \[ \left| \sum_{i} a_{i}^{3} \right| \leqslant \frac{n-2}{\sqrt{n ( n-1 )}}
     \left( \sum_{i} a_{i}^{2} \right)^{\frac{3}{2}} , \]
  where equality holds if and only if at least $n-1$ numbers of $a_{i}$ are
  equal.
\end{proposition}

\section{Preservation of curvature pinching}

Let $F: M \times [ 0,T ) \rightarrow \mathbb{H}^{n+q} ( c )$ be a
mean curvature flow in the hyperbolic space $\mathbb{H}^{n+q} ( c
)$. Let $M_{t} =F ( M,t )$. Suppose that $M_{0}$ is an
$n$-dimensional $( n \ge 6 )$ complete submanifold satisfying $\sup
( | h |^{2} - \alpha ( n, | H | ,c ) ) <0$ and $| H |^{2} +n^{2}
c>0$.

The evolution equations of the mean curvature flow take the same form as in
{\cite{baker2011mean,MR3078951}}.

\begin{lemma}
  \label{evo}For the mean curvature flow $F: M \times [ 0,T ) \rightarrow
  \mathbb{H}^{n+q} ( c )$, we have
  \begin{enumerateroman}
    \item $\dt | h |^{2} = \Delta | h |^{2} -2 | \nabla h |^{2} +2R_{1} +4 c |
    H |^{2} -2 n c | h |^{2}$,

    \item $\dt | H |^{2} = \Delta | H |^{2} -2 | \nabla H |^{2} +2 R_{2} +2n c
    | H |^{2}$,

    \item $\dt \normho^{2} = \Delta \normho^{2} -2 | \nabla h |^{2} +
    \frac{2}{n} | \nabla H |^{2} +2R_{1} - \frac{2}{n} R_{2} -2 n c
    \normho^{2}$.
  \end{enumerateroman}
\end{lemma}

We define a function $\mathring{\alpha} : ( -n^{2} c,+ \infty ) \rightarrow
\mathbb{R}$ by
\begin{equation}
  \mathring{\alpha} ( y ) =n c+ \frac{n^{2} -2n+2}{2 ( n-1 ) n}  y-
  \frac{n-2}{2 ( n-1 )} \sqrt{y^{2} +4 ( n-1 ) c y} .
\end{equation}
It's obvious that $\mathring{\alpha} ( | H |^{2} ) = \alpha ( n, | H | ,c ) -
\frac{1}{n} | H |^{2}$. Moreover, we have the following lemma.

\begin{lemma}
  $\label{app}$For $n \ge 6$, $c<0$ and $y>-n^{2} c$,
  $\mathring{\alpha}$ has the following properties.
  \begin{enumerateroman}
    \item $y  \mathring{\alpha}' ( y ) \cdot \left( \mathring{\alpha} ( y ) +
    \frac{1}{n} y+n c \right) \equiv \mathring{\alpha} ( y ) \cdot \left(
    \mathring{\alpha} ( y ) + \frac{1}{n} y-n c \right)$,

    \item $\frac{n-2}{\sqrt{n ( n-1 )}} \sqrt{y  \mathring{\alpha} ( y )}
    \equiv \frac{1}{n} y- \mathring{\alpha} ( y ) +n c$,

    \item $0< \mathring{\alpha} ( y ) < \frac{y+n^{2} c}{n ( n-1 )}$, \ $0<
    \mathring{\alpha}' ( y ) < \frac{1}{n ( n-1 )}$, \ $\mathring{\alpha}'' (
    y ) >0$,

    \item $2 \sqrt{y}   \mathring{\alpha}' ( y ) < \sqrt{\mathring{\alpha} ( y
    )}$,

    \item $y  \mathring{\alpha}' ( y ) > \mathring{\alpha} ( y )$,

    \item $2y  \mathring{\alpha}'' ( y ) + \mathring{\alpha}' ( y ) < \frac{2
    ( n-1 )}{n ( n+2 )}$.
  \end{enumerateroman}
\end{lemma}

\begin{proof}
  By direct computations, we get
  \[ \mathring{\alpha}' ( y ) = \frac{n^{2} -2n+2}{2 ( n-1 ) n} - \frac{n-2}{2
     ( n-1 )}   \frac{y+2 ( n-1 ) c}{\sqrt{y^{2} +4 ( n-1 ) c y}} ,
     \hspace{1em} \mathring{\alpha}'' ( y ) = \frac{2 ( n-1 ) ( n-2 ) c^{2}}{(
     y^{2} +4 ( n-1 ) c y )^{3/2}} . \]
  \[  \]
  We use a variable substitution $\xi = \frac{y}{\sqrt{y^{2} +4 ( n-1 ) c y}}$
  to simplify formulas. Then $y>-n^{2} c$ implies $1< \xi < \frac{n}{n-2}$.
  Hence $y= \frac{4 ( n-1 ) c}{\xi^{-2} -1}$ and $\sqrt{y^{2} +4 ( n-1 ) c y}
  = \frac{4 ( n-1 ) c}{\xi^{-1} - \xi}$.

  With this variable substitution, one can verify (i) and (ii) easily.

  For the rest, we have
  \[ \frac{y+n^{2} c}{\mathring{\alpha} ( y )} = \frac{2n^{2}}{n- ( n-2 ) \xi}
   -n>n ( n-1 ) , \]
  \[ \frac{2 \sqrt{y}   \mathring{\alpha}' ( y )}{\sqrt{\mathring{\alpha} ( y
     )}} = \frac{n ( \xi -1 ) +2}{\sqrt{n ( n-1 )}} < \frac{4}{n-2}
     \sqrt{\frac{n-1}{n}} <1, \]
  \[ y  \mathring{\alpha}' ( y ) - \mathring{\alpha} ( y ) = [ ( n-2 ) \xi -n ]
     c>0, \]
  and
  \[ 2y  \mathring{\alpha}'' ( y ) + \mathring{\alpha}' ( y ) + \frac{1}{n} =
     \frac{n}{2 ( n-1 )} + \frac{n-2}{4 ( n-1 )} ( \xi^{3} -3 \xi ) < \frac{n}{(
     n-2 )^{2}} \leqslant \frac{3}{n+2} . \]

\end{proof}

For convenience, we denote $\mathring{\alpha} ( | H |^{2} )$,
$\mathring{\alpha}' ( | H |^{2} )$ and $\mathring{\alpha}'' ( | H
|^{2} )$ by $\mathring{\alpha}$, $\mathring{\alpha}'$ and
$\mathring{\alpha}''$, respectively. Then we get the evolution
equation of $\mathring{\alpha}$.
\begin{equation}
  \dt \mathring{\alpha} = \Delta \mathring{\alpha} +2 \mathring{\alpha}' \cdot
  ( - | \nabla H |^{2} +R_{2} +n c | H |^{2} ) - \mathring{\alpha}'' \cdot |
  \nabla | H |^{2} |^{2} .
\end{equation}

It's seen from Theorem 1 of {\cite{MR1458750}} that $M_{0}$ is
compact. Hence there exists a small positive number $\varepsilon$,
such that $M_{0}$ satisfies
\begin{equation}
  \normho^{2} < \mathring{\alpha} - \varepsilon   \omega , \hspace{2em}
  \tmop{where} \hspace{1em} \omega = | H |^{2} +4 ( n-1 ) c.
\end{equation}
In the following we prove that the pinching condition above is
preserved along the flow.

\begin{theorem}
  \label{pinch}If $M_{0}$ satisfies $\normho^{2} < \mathring{\alpha} -
  \varepsilon   \omega$ and $| H |^{2} +n^{2} c>0$, then this condition holds
  for all time $t \in [ 0,T )$.
\end{theorem}

\begin{proof}
  Suppose $\normho^{2} < \mathring{\alpha} - \varepsilon   \omega$ remains
  true for $t \in [ 0, \tau )$. Note that $\mathring{\alpha} \rightarrow 0$ as $|
  H |^{2} \rightarrow -n^{2} c$. This implies that $| H |^{2} +n^{2} c>0$ also remains true
  for $t \in [ 0, \tau )$. On the time interval $[ 0, \tau )$, we have the following
  evolution equation for $U= \normho^{2} - \mathring{\alpha} + \varepsilon
  \omega$.
  \begin{eqnarray}
    \left( \dt - \Delta \right) U & = & -2 | \nabla h |^{2} + \frac{2}{n} |
    \nabla H |^{2} +2 \left( \mathring{\alpha}' - \varepsilon \right) | \nabla
    H |^{2} + \mathring{\alpha}'' \cdot | \nabla | H |^{2} |^{2} \nonumber\\
    &  & +2R_{1} - \frac{2}{n} R_{2} -2 n c \normho^{2} -2 \left(
    \mathring{\alpha}' - \varepsilon \right) ( R_{2} +n c | H |^{2} ) .
    \label{dtU}
  \end{eqnarray}
  By Lemma \ref{dA2} and Lemma \ref{app} (vi), the first line of the right
  hand side of (\ref{dtU}) is not greater than
  \[ \left[ - \frac{2 ( n-1 )}{n ( n+2 )} + \mathring{\alpha}' +2 | H |^{2}
     \mathring{\alpha}''   \right] | \nabla H |^{2} \le 0. \]
  By Lemma \ref{R12}, the second line of the right hand side of (\ref{dtU}) is
  not greater than
  \begin{eqnarray*}
    &  & 2 \normho^{2} \left( \normho^{2} + \frac{1}{n} | H |^{2} - n c
    \right) +2P_{2} \left( 2 \normho^{2} - \frac{1}{n} | H |^{2} \right)\\
    &  & -2 \left( \mathring{\alpha}' - \varepsilon \right)   | H |^{2}
    \left( \normho^{2} + \frac{1}{n} | H |^{2} +n c \right) +2 \left(
    \mathring{\alpha}' - \varepsilon \right)   | H |^{2} P_{2} .
  \end{eqnarray*}
  Replacing $\normho^{2}$ by $U+ \mathring{\alpha} - \varepsilon   \omega$,
  the right hand side of the inequality above becomes
  \begin{eqnarray}
    &  & 2U \left[ 2 \mathring{\alpha} + \frac{1}{n} | H |^{2} -n c -
    \mathring{\alpha}' | H |^{2}  +2P_{2} + \varepsilon ( | H |^{2} -2 \omega
    ) \right] +2U^{2} \nonumber\\
    &  & +2 \left[ \mathring{\alpha} \cdot \left( \mathring{\alpha} +
    \frac{1}{n} | H |^{2} -n c \right) - \mathring{\alpha}' | H |^{2} \cdot
    \left( \mathring{\alpha} + \frac{1}{n} | H |^{2} +n c \right) \right]
    \nonumber\\
    &  & +2P_{2} \left[ 2 \mathring{\alpha} - \frac{1}{n} | H |^{2} + | H
    |^{2} \mathring{\alpha}' - \varepsilon ( | H |^{2} +2 \omega ) \right]
    \label{Uandneg}\\
    &  & +2 \varepsilon   \omega \left[ - \left( 2 \mathring{\alpha} +
    \frac{1}{n} | H |^{2} -n c- \mathring{\alpha}' | H |^{2} \right) + \frac{|
    H |^{2}}{\omega} \left( \mathring{\alpha} + \frac{1}{n} | H |^{2} +n c
    \right) \right] \nonumber\\
    &  & -2 \varepsilon^{2}   \omega ( | H |^{2} - \omega ) . \nonumber
  \end{eqnarray}
  By Lemma \ref{app} (i) and (iii), the expression in the second square
  bracket of the RHS of (\ref{Uandneg}) equals zero, and the expression in the third square
  bracket of the RHS of (\ref{Uandneg}) is negative. By a direct computation, the
  expression in the last square bracket of the RHS of (\ref{Uandneg}) equals
  \[ \left( \frac{3 ( n-2 ) | H |^{2}}{\sqrt{| H |^{4} +4 ( n-1 ) c  | H
     |^{2}}} -n \right) c, \]
  which is negative. So, we have
  \[ \left( \dt - \Delta \right) U<2U \left[ 2 \mathring{\alpha} + \frac{1}{n}
     | H |^{2} -n c - \mathring{\alpha}' | H |^{2}  +2P_{2} + \varepsilon ( |
     H |^{2} -2 \omega ) \right] +2U^{2} . \]
  By the maximum principle, the assertion follows.
\end{proof}

From the preservation of $\mathring{\alpha} > \varepsilon   \omega$, we have

\begin{corollary}
  There exists a positive constant $\delta$ depending on $\varepsilon$, such
  that $| H |^{2} +n^{2} c> \delta$ holds for all time $t \in [ 0,T )$.
\end{corollary}

For completeness, we present a new proof of the following
proposition {\cite{MR3078951}}, which states that the maximal
existence time is finite.

\begin{proposition} Let $F: M \times [ 0,T ) \rightarrow \mathbb{H}^{n+q} ( c )$ be a
mean
  curvature flow. If the initial value $M_{0}$ is a closed submanifold, then
  the maximal existence time $T$ is finite.
\end{proposition}

\begin{proof}
  We use the Minkowski model for hyperbolic spaces. Let $\mathbb{R}^{1,m}$ be
  the $m+1$ dimensional Minkowski space. For $X,Y \in \mathbb{R}^{1,m}$, where $X=
  ( x_{0} , \cdots ,x_{m} )$, $Y= ( y_{0} , \cdots ,y_{m} )$, the inner
  product in $\mathbb{R}^{1,m}$ is defined by
  \[ \langle X,Y \rangle =-x_{0} y_{0} +x_{1} y_{1} + \cdots +x_{m} y_{m} . \]
  For $c<0$, we consider the following spacelike hypersurface in
  $\mathbb{R}^{1,m}$
  \[ -x_{0}^{2} +x_{1}^{2} + \cdots +x_{m}^{2} =1/c, \hspace{1em} x_{0}
     \ge 1/ \sqrt{-c} . \]
  It has constant sectional curvature $c$. We identify $\mathbb{H}^{m} ( c )$
  with this hypersurface.

  Let $X:M^{n} \rightarrow \mathbb{H}^{m} ( c ) \subset \mathbb{R}^{1,m}$
  be a submanifold immersed in the hyperbolic space $\mathbb{H}^{m} ( c )$. We denote by $\nabla$ and
  $\bar{\nabla}$ the Levi-Civita connections of $M$ and $\mathbb{H}^{m} ( c )$,
  respectively. Let $u$, $v$ be tangent vector fields over $M$. Since $\langle
  X,X \rangle =1/c$, we have $\langle v X,X \rangle =0$ and $\langle u v X,X
  \rangle =- \langle u X,v X \rangle =- \langle u,v \rangle$. Thus $u v X=
  \bar{\nabla}_{u} v-c \langle u,v \rangle X$. Then we have
  \begin{eqnarray}
    \nabla^{2}_{u,v} X & = & u v X- ( \nabla_{u} v ) X \nonumber\\
    & = & \bar{\nabla}_{u} v-c \langle u,v \rangle X- \nabla_{u} v \\
    & = & h ( u,v ) -c \langle u,v \rangle X. \nonumber
  \end{eqnarray}
  Taking the trace of the both sides of (3.6), we obtain $\Delta X=H-n c X$.

  Let $X:M \times [ 0,T ) \rightarrow \mathbb{H}^{n+q} ( c ) \subset
  \mathbb{R}^{1,n+q}$ be a mean curvature flow. Then the equation of mean
  curvature flow becomes $\dt X= \Delta X+n c X$. Particularly, $\dt x_{0} =
  \Delta x_{0} +n c x_{0}$. By the maximum principle, we have $x_{0} \le
  \sup_{t=0} ( x_{0} ) \cdot \mathe^{n c t}$. Therefore, $T$ is finite.
\end{proof}

Since the maximal existence time is finite, we have $\max_{M_{t}} |
h |^{2} \rightarrow \infty$ as $t \rightarrow T$. This can be shown
by using analogous argument in the proof of the corresponding
theorem in {\cite{MR2739807}}. Once $| h |^{2}$ is uniformly
bounded, then all higher derivatives $| \nabla^{m} h |^{2}$ are
uniformly bounded. Hence the solution $M_{t}$ converge to a limit
$M_{T}$ in $C^{\infty}$--topology as $t \rightarrow T$. Thus, the
flow can be extended over time $T$. This contradicts the maximality
of $T$.

\section{An estimate for traceless second fundamental form}

In this section, we derive an estimate for the traceless second fundamental
form, which shows $\normho$ grows slower than $| H |$ along the mean curvature
flow.

\begin{theorem}
  \label{sa0h2}If $M_{0}$ satisfies $\normho^{2} < \mathring{\alpha} -
  \varepsilon   \omega$ and $| H |^{2} +n^{2} c>0$, then there exist constants
  $0< \sigma <1$ and $C_{0} >0$ depending only on $M_{0}$, such that for all
  $t \in [ 0,T )$ we have
  \[ \normho^{2} \le C_{0}   | H |^{2 ( 1- \sigma )} . \]
\end{theorem}

To prove Theorem 4.1, we need to study the auxiliary function:
\[ f_{\sigma} = \frac{\normho^{2}}{\mathring{\alpha}^{1- \sigma}} ,
   \hspace{1em} 0< \sigma <1. \]

First, we derive the evolution equation of $f_{\sigma}$.

\begin{lemma}
  \label{ptf} Along the mean curvature flow, we have
  \[ \dt f_{\sigma} \le \Delta f_{\sigma} + \frac{2}{\normho} | \nabla
     f_{\sigma} | | \nabla H | -12  \varepsilon \frac{f_{\sigma}}{\normho^{2}}
     | \nabla H |^{2} -4 c f_{\sigma} + \sigma | H |^{2} f_{\sigma} . \]
\end{lemma}

\begin{proof}
  By a direct computation, we have
  \begin{equation} \dt f_{\sigma} =f_{\sigma} \left( \frac{\dt \normho^{2}}{\normho^{2}} - (
     1- \sigma ) \frac{\dt \mathring{\alpha}}{\mathring{\alpha}} \right) . \label{partialtf} \end{equation}
  The gradient of $f_{\sigma}$ can be written as
  \[ \nabla f_{\sigma} =f_{\sigma} \left( \frac{\nabla
     \normho^{2}}{\normho^{2}} - ( 1- \sigma ) \frac{\nabla
     \mathring{\alpha}}{\mathring{\alpha}} \right). \]
  The Laplacian of $f_{\sigma}$ is given by
  \begin{equation}
    \Delta f_{\sigma} =f_{\sigma} \left( \frac{\Delta
    \normho^{2}}{\normho^{2}} - ( 1- \sigma ) \frac{\Delta
    \mathring{\alpha}}{\mathring{\alpha}} \right) -2 ( 1- \sigma )
    \frac{\left\langle \nabla f_{\sigma} , \nabla \mathring{\alpha}
    \right\rangle}{\mathring{\alpha}} + \sigma ( 1- \sigma ) f_{\sigma}
    \frac{\left| \nabla \mathring{\alpha} \right|^{2}}{\left|
    \mathring{\alpha} \right|^{2}} . \label{lapf}
  \end{equation}
  By (\ref{partialtf}), (\ref{lapf}) and the evolution equations, we get
  \begin{eqnarray}
    \left( \dt - \Delta \right) f_{\sigma} & = & 2 ( 1- \sigma )
    \frac{\left\langle \nabla f_{\sigma} , \nabla \mathring{\alpha}
    \right\rangle}{\mathring{\alpha}} - \sigma ( 1- \sigma ) f_{\sigma}
    \frac{\left| \nabla \mathring{\alpha} \right|^{2}}{\left|
    \mathring{\alpha} \right|^{2}} \nonumber\\
    &  & + \frac{2 f_{\sigma}}{\normho^{2}} \left( \frac{| \nabla H |^{2}}{n}
    - | \nabla h |^{2} \right) + ( 1- \sigma )
    \frac{f_{\sigma}}{\mathring{\alpha}} \left( 2 \mathring{\alpha}' | \nabla
    H |^{2} + \mathring{\alpha}''   | \nabla | H |^{2} |^{2} \right)
    \nonumber\\
    &  & +2 f_{\sigma} \left[ \frac{1}{\normho^{2}} \left( R_{1} -
    \frac{1}{n}  R_{2} \right) -n c- ( 1- \sigma )
    \frac{\mathring{\alpha}'}{\mathring{\alpha}} (  R_{2} +n c | H |^{2} )
    \right] \nonumber\\
    & \le & \frac{2}{\mathring{\alpha}} | \nabla f_{\sigma} |   \left|
    \nabla \mathring{\alpha} \right|  \label{dtf}\\
    &  & +f_{\sigma} \left[   \frac{2}{\normho^{2}} \left( \frac{| \nabla H
    |^{2}}{n} - | \nabla h |^{2} \right) + \frac{2
    \mathring{\alpha}'}{\mathring{\alpha}} | \nabla H |^{2} +
    \frac{\mathring{\alpha}''}{\mathring{\alpha}} | \nabla | H |^{2} |^{2}
    \right] \nonumber\\
    &  & +2 f_{\sigma} \left[ \frac{1}{\normho^{2}} \left( R_{1} -
    \frac{1}{n}  R_{2} \right) -n c- ( 1- \sigma )
    \frac{\mathring{\alpha}'}{\mathring{\alpha}} (  R_{2} +n c | H |^{2} )
    \right] . \nonumber
  \end{eqnarray}
By Lemma \ref{dA2} (ii) and Lemma \ref{app} (iv), we have
  \begin{equation}
    \frac{1}{\mathring{\alpha}}   \left| \nabla \mathring{\alpha} \right| =
    \frac{\mathring{\alpha}'}{\mathring{\alpha}} | \nabla | H |^{2} |
    \le \frac{1}{\normho} | \nabla H | . \label{aopaodfsdh2}
  \end{equation}
  Now we estimate the expression in the first square bracket of the right hand
  side of (\ref{dtf}). From Lemma \ref{dA2} and Lemma \ref{app}, we have
  \begin{eqnarray*}
    &  & \frac{2}{\normho^{2}} \left( \frac{| \nabla H |^{2}}{n} - | \nabla h
    |^{2} \right) + \frac{2  \mathring{\alpha}'}{\mathring{\alpha}} | \nabla H
    |^{2} + \frac{\mathring{\alpha}''}{\mathring{\alpha}} | \nabla | H |^{2}
    |^{2}\\
    & \le & \left( \frac{4 ( 1-n )}{n ( n+2 )}   \frac{1}{\normho^{2}}
    + \frac{2  \mathring{\alpha}'}{\mathring{\alpha}} +4 | H |^{2}
    \frac{\mathring{\alpha}''}{\mathring{\alpha}} \right) | \nabla H |^{2}\\
    & = & \left[ \frac{4 ( 1-n )}{n ( n+2 )}   \left(
    \frac{1}{\mathring{\alpha}} +  \frac{\mathring{\alpha} -
    \normho^{2}}{\mathring{\alpha} \normho^{2}} \right) + \frac{2
    \mathring{\alpha}'}{\mathring{\alpha}} +4 | H |^{2}
    \frac{\mathring{\alpha}''}{\mathring{\alpha}} \right] | \nabla H |^{2}\\
    & \le & \left[ - \frac{4 ( n-1 )}{n ( n+2 )}   \frac{\varepsilon
    \omega}{\mathring{\alpha} \normho^{2}} + \frac{1}{\mathring{\alpha}}
    \left( \frac{4 ( 1-n )}{n ( n+2 )} +2  \mathring{\alpha}' +4 | H |^{2}
    \mathring{\alpha}'' \right) \right] | \nabla H |^{2}\\
    & \le & - \frac{4 ( n-1 )}{n ( n+2 )}
    \frac{\omega}{\mathring{\alpha}}   \frac{\varepsilon}{\normho^{2}} |
    \nabla H |^{2}\\
    & \le & - \frac{12  \varepsilon}{\normho^{2}} | \nabla H |^{2} .
  \end{eqnarray*}

  Next we estimate the expression in the second square bracket of the right
  hand side of (\ref{dtf}). By Lemma \ref{R12}, we have
  \begin{eqnarray}
    &  & \frac{1}{\normho^{2}} \left( R_{1} - \frac{1}{n}  R_{2} \right) -n
    c- ( 1- \sigma ) \frac{\mathring{\alpha}'}{\mathring{\alpha}} (  R_{2} +n
    c | H |^{2} ) \nonumber\\
    & \le & | h |^{2} -n c- ( 1- \sigma )
    \frac{\mathring{\alpha}'}{\mathring{\alpha}} | H |^{2} (   | h |^{2} +n c
    ) \nonumber\\
    &  & +P_{2} \left[ \frac{1}{\normho^{2}} \left( 2 \normho^{2} -
    \frac{1}{n} | H |^{2} \right) + ( 1- \sigma )
    \frac{\mathring{\alpha}'}{\mathring{\alpha}} | H |^{2} \right]
    \label{sH22c}\\
    & \le & \sigma ( | h |^{2} +n c ) + ( 1- \sigma ) ( | h |^{2} +n c
    ) \left( 1- \frac{  \mathring{\alpha}'}{\mathring{\alpha}} | H |^{2}
    \right) -2n c \nonumber\\
    &  & +P_{2} \left[ 2+ \left( - \frac{1}{n} + \mathring{\alpha}' \right)
    \frac{| H |^{2}}{\mathring{\alpha}} \right] . \nonumber
  \end{eqnarray}
  From Lemma \ref{app} (iii), we get $2+ \left( - \frac{1}{n} +
  \mathring{\alpha}' \right) \frac{| H |^{2}}{\mathring{\alpha}} <0$ and $1-
  \frac{  \mathring{\alpha}'}{\mathring{\alpha}} | H |^{2} <0$. Then we have
  \begin{eqnarray*}
    ( | h |^{2} +n c ) \left( 1- \frac{
    \mathring{\alpha}'}{\mathring{\alpha}} | H |^{2} \right) & \le &
    \left( \frac{1}{n} | H |^{2} +n c \right) \left( 1- \frac{
    \mathring{\alpha}'}{\mathring{\alpha}} | H |^{2} \right)\\
    & = & \left( \frac{( n-2 ) | H |^{2}}{\sqrt{| H |^{4} +4 ( n-1 ) c  | H
    |^{2}}} +n \right) c\\
    & < & ( 2n-2 ) c.
  \end{eqnarray*}
  Therefore, the right hand side of (\ref{sH22c}) is less than
  \begin{eqnarray*}
    \sigma ( | h |^{2} +n c ) + ( 1- \sigma ) ( 2n-2 ) c-2n c & = & \sigma ( |
    h |^{2} + ( 2-n ) c ) -2 c\\
    & < & \sigma \left( \mathring{\alpha} + \frac{1}{n} | H |^{2} + ( 2-n ) c
    \right) -2 c\\
    & < & \frac{\sigma}{2} | H |^{2} -2c.
  \end{eqnarray*}
This proves Lemma 4.2.
\end{proof}

To estimate the term $\sigma | H |^{2}
f_{\sigma}$ in Lemma \ref{dtf}, we need the following.

\begin{lemma}
  \label{lapa0}If a submanifold in $\mathbb{H}^{n+q} ( c )$ satisfies
  $\normho^{2} < \mathring{\alpha} - \varepsilon   \omega$ and $| H |^{2}
  +n^{2} c>0$, then we have
  \[ \Delta \normho^{2} \ge 2 \left\langle \mathring{h} , \nabla^{2} H
     \right\rangle +  \frac{\varepsilon}{2}   | H |^{2}   \normho^{2} . \]
\end{lemma}

\begin{proof}
  From (\ref{lapho2}) and Lemma \ref{dA2} (i), we only need to prove $W\ge
  \frac{\varepsilon}{4}  H^{2}   \normho^{2}$.

  We work with the special local orthonormal frame, such that $\nu_{1} =H/ | H
  |$ and $\mathring{h}^{1} = \tmop{diag} \left( \mathring{\lambda}_{1} ,
  \cdots , \mathring{\lambda}_{n} \right)$. Then we expand $W$ to get
  \begin{eqnarray*}
    W & = & n c \normho^{2} -P_{1}^{2} + \frac{1}{n} \normho^{2} | H |^{2}\\
    &  & -2 \sum_{\alpha >1} \left( \sum_{i} \mathring{\lambda}_{i}
    \mathring{h}^{\alpha}_{i i} \right)^{2} -2
    \sum_{\tmscript{\begin{array}{c}
      \alpha >1\\
      i \neq j
    \end{array}}} \left( \left( \mathring{\lambda}_{i} -
    \mathring{\lambda}_{j} \right) \mathring{h}^{\alpha}_{i j} \right)^{2}\\
    &  & - \sum_{\alpha , \beta >1} \left( \sum_{i,j}
    \mathring{h}^{\alpha}_{i j}   \mathring{h}^{\beta}_{i j} \right)^{2} -
    \sum_{\tmscript{\begin{array}{c}
      \alpha , \beta >1\\
      i,j
    \end{array}}} \left( \sum_{k} \left( \mathring{h}^{\alpha}_{i k}
    \mathring{h}^{\beta}_{j k} - \mathring{h}^{\alpha}_{j k}
    \mathring{h}^{\beta}_{i k} \right) \right)^{2}\\
    &  & + | H | \sum_{\alpha ,i} \mathring{\lambda}_{i} \left(
    \mathring{h}^{\alpha}_{i i} \right)^{2} + | H |
    \sum_{\tmscript{\begin{array}{c}
      \alpha >1\\
      i \neq j
    \end{array}}} \mathring{\lambda}_{i} \left( \mathring{h}^{\alpha}_{i j}
    \right)^{2} .
  \end{eqnarray*}

  Using (\ref{P1Q1}) and (\ref{sP232}) again, we have
  \[ \sum_{\alpha >1} \left( \sum_{i} \mathring{\lambda}_{i}
     \mathring{h}^{\alpha}_{i i} \right)^{2} \leqslant P_{1} Q_{1} , \]
  \[ \sum_{\alpha , \beta >1} \left( \sum_{i,j} \mathring{h}^{\alpha}_{i j}
     \mathring{h}^{\beta}_{i j} \right)^{2} + \sum_{\tmscript{\begin{array}{c}
       \alpha , \beta >1\\
       i,j
     \end{array}}} \left( \sum_{k} \left( \mathring{h}^{\alpha}_{i k}
     \mathring{h}^{\beta}_{j k} - \mathring{h}^{\alpha}_{j k}
     \mathring{h}^{\beta}_{i k} \right) \right)^{2} \leqslant \frac{3}{2}
     P_{2}^{2} . \]

  By Proposition \ref{lao3}, we have
  \begin{eqnarray*}
    | H | \sum_{\alpha ,i} \mathring{\lambda}_{i} \left(
    \mathring{h}^{\alpha}_{i i} \right)^{2} & \geqslant & -
    \tfrac{n-2}{\sqrt{n ( n-1 )}} | H | \sqrt{P_{1}} \left( \normho^{2} -Q_{2}
    \right)\\
    & \geqslant & -  \tfrac{n-2}{\sqrt{n ( n-1 )}} | H | \left( \frac{1}{2}
    \left( P_{1} + \normho^{2} \right) \normho - \sqrt{P_{1}} Q_{2} \right)\\
    & = & -  \tfrac{n-2}{\sqrt{n ( n-1 )}} | H | \left( \normho^{3} -
    \frac{1}{2} \normho P_{2} - \sqrt{P_{1}} Q_{2} \right) .
  \end{eqnarray*}

  We estimate the rest terms as following.
  \begin{eqnarray*}
    &  & | H | \sum_{\tmscript{\begin{array}{c}
      \alpha >1\\
      i \neq j
    \end{array}}} \mathring{\lambda}_{i} \left( \mathring{h}^{\alpha}_{i j}
    \right)^{2} -2 \sum_{\tmscript{\begin{array}{c}
      \alpha >1\\
      i \neq j
    \end{array}}} \left( \left( \mathring{\lambda}_{i} -
    \mathring{\lambda}_{j} \right) \mathring{h}^{\alpha}_{i j} \right)^{2}\\
    & = & \sum_{\tmscript{\begin{array}{c}
      \alpha >1\\
      i \neq j
    \end{array}}} \left[ \frac{1}{2} | H | \left( \mathring{\lambda}_{i} +
    \mathring{\lambda}_{j} \right) +2 \left( \mathring{\lambda}_{i} +
    \mathring{\lambda}_{j} \right)^{2} -4 \left( \mathring{\lambda}_{i}^{2} +
    \mathring{\lambda}_{j}^{2} \right) \right] \left( \mathring{h}^{\alpha}_{i
    j} \right)^{2}\\
    & \geqslant & \sum_{\tmscript{\begin{array}{c}
      \alpha >1\\
      i \neq j
    \end{array}}} \left[ \frac{1}{2} | H | \left( \mathring{\lambda}_{i} +
    \mathring{\lambda}_{j} \right) + \frac{3}{4} \left( \mathring{\lambda}_{i}
    + \mathring{\lambda}_{j} \right)^{2} -4P_{1} \right] \left(
    \mathring{h}^{\alpha}_{i j} \right)^{2} .
  \end{eqnarray*}
  Letting $y= \mathring{\lambda}_{i} + \mathring{\lambda}_{j}$, we have $y
  \geqslant - \sqrt{2 ( \mathring{\lambda}_{i}^{2} +
  \mathring{\lambda}_{j}^{2} )} \geqslant - \sqrt{2P_{1}}$. By Lemma
  \ref{app} (iii), we get $\sqrt{P_{1}} < \sqrt{\mathring{\alpha}} < \frac{| H
  |}{\sqrt{n ( n-1 )}}$. Then the function $\frac{3}{4} y^{2} + \frac{1}{2} |
  H | y$ is increasing for $y \geqslant - \sqrt{2P_{1}}$. Thus $\frac{3}{4}
  y^{2} + \frac{1}{2} | H | y \geqslant \frac{3}{2} P_{1} - \frac{\sqrt{2}}{2}
  | H | \sqrt{P_{1}}$. We obtain
  \begin{equation*}
    | H | \sum_{\tmscript{\begin{array}{c}
      \alpha >1\\
      i \neq j
    \end{array}}} \mathring{\lambda}_{i} ( \mathring{h}^{\alpha}_{i j}
    )^{2} -2 \sum_{\tmscript{\begin{array}{c}
      \alpha >1\\
      i \neq j
    \end{array}}} ( ( \mathring{\lambda}_{i} -
    \mathring{\lambda}_{j} ) \mathring{h}^{\alpha}_{i j} )^{2}
    \geqslant - \left( \tfrac{\sqrt{2}}{2} | H | \sqrt{P_{1}} + \tfrac{5}{2}
    P_{1} \right) Q_{2} .
  \end{equation*}
Applying these estimates, we get
  \begin{eqnarray*}
    W & \geqslant & n c \normho^{2} -P_{1}^{2} + \frac{1}{n} \normho^{2} | H
    |^{2} -2P_{1} Q_{1} - \frac{3}{2} P_{2}^{2} - \frac{5}{2} P_{1} Q_{2} -
    \frac{\sqrt{2}}{2} | H | \sqrt{P_{1}} Q_{2}\\
    &  & - \tfrac{n-2}{\sqrt{n ( n-1 )}} | H | \left( \normho^{3} -
    \frac{1}{2} \normho P_{2} - \sqrt{P_{1}} Q_{2} \right)\\
    & = & n c \normho^{2} - \normho^{4} + \frac{1}{n} \normho^{2} | H |^{2} -
    \tfrac{n-2}{\sqrt{n ( n-1 )}} | H | \normho^{3}\\
    &  & + \frac{1}{2} \left( \tfrac{n-2}{\sqrt{n ( n-1 )}} | H | \normho
    P_{2} -P_{2}^{2} -P_{1} Q_{2} \right) + \left( \tfrac{n-2}{\sqrt{n ( n-1
    )}} - \frac{\sqrt{2}}{2} \right) | H | \sqrt{P_{1}} Q_{2} .
  \end{eqnarray*}
  Since $P_{2}^{2} +P_{1} Q_{2} \leqslant \normho^{2} P_{2} \leqslant
  \tfrac{1}{\sqrt{n ( n-1 )}} | H | \normho P_{2}$, we have
  \begin{eqnarray*}
    W & \geqslant & \normho^{2} \left( n c- \normho^{2} + \frac{1}{n} | H
    |^{2} - \tfrac{n-2}{\sqrt{n ( n-1 )}} | H | \normho \right)\\
    & \geqslant & \normho^{2} \left( n c- \left( \mathring{\alpha} -
    \varepsilon   \omega \right) + \frac{1}{n} | H |^{2} -
    \tfrac{n-2}{\sqrt{n ( n-1 )}} | H | \sqrt{\mathring{\alpha}} \right)\\
    & = & \normho^{2}   \varepsilon   \omega\\
    & \geqslant & \frac{\varepsilon}{4} | H |^{2}   \normho^{2} .
  \end{eqnarray*}

\end{proof}

From (\ref{lapf}), (\ref{aopaodfsdh2}) and Lemma \ref{lapa0}, we have
\begin{eqnarray*}
  \Delta f_{\sigma} & \ge & \frac{f_{\sigma} \Delta
  \normho^{2}}{\normho^{2}} -  ( 1- \sigma )
  \frac{f_{\sigma}}{\mathring{\alpha}} \Delta \mathring{\alpha} -2 ( 1- \sigma
  ) \frac{\left\langle \nabla f_{\sigma} , \nabla \mathring{\alpha}
  \right\rangle}{\mathring{\alpha}}\\
  & \ge & \frac{2f_{\sigma}}{\normho^{2}}   \left\langle \mathring{h} ,
  \nabla^{2} H \right\rangle + \frac{\varepsilon}{2} | H |^{2} f_{\sigma} -  (
  1- \sigma ) \frac{f_{\sigma}}{\mathring{\alpha}} \Delta \mathring{\alpha} -
  \frac{2}{\normho} | \nabla f_{\sigma} | | \nabla H |  .
\end{eqnarray*}
This is equivalent to
\begin{equation} \frac{\varepsilon}{2} | H |^{2} f_{\sigma} \le \Delta f_{\sigma} -
   \frac{2f_{\sigma}}{\normho^{2}}   \left\langle \mathring{h} , \nabla^{2} H
   \right\rangle +  ( 1- \sigma ) \frac{f_{\sigma}}{\mathring{\alpha}} \Delta
   \mathring{\alpha} + \frac{2}{\normho} | \nabla f_{\sigma} | | \nabla H | .
\label{inte1} \end{equation}
We multiply both sides of this inequality by $f_{\sigma}^{p-1}$, then
integrate them over $M_{t}$. From the divergence theorem and the relation
$\nabla_{i} \mathring{h}_{i j} = \frac{n-1}{n}   \nabla_{j} H$, we have
\begin{equation} \int_{M_{t}} f_{\sigma}^{p-1} \Delta f_{\sigma} \mathd \mu_{t} =- ( p-1 )
   \int_{M_{t}} f_{\sigma}^{p-2} | \nabla f_{\sigma} |^{2} \mathd \mu_{t}
   \le 0, \end{equation}
\begin{eqnarray}
  - \int_{M_{t}} \frac{f_{\sigma}^{p}}{\normho^{2}}   \left\langle
  \mathring{h} , \nabla^{2} H \right\rangle \mathd \mu_{t} & = & -
  \int_{M_{t}} \frac{f_{\sigma}^{p-1}}{\mathring{\alpha}^{1- \sigma}}
  \mathring{h}^{\alpha}_{i j} \nabla^{2}_{i,j} H^{\alpha}   \mathd \mu_{t} \nonumber \\
  & = & \int_{M_{t}} \nabla_{i} \left(
  \frac{f_{\sigma}^{p-1}}{\mathring{\alpha}^{1- \sigma}} \mathring{h}_{i
  j}^{\alpha} \right) \nabla_{j} H^{\alpha}   \mathd \mu_{t} \nonumber \\
  & = & \int_{M_{t}} \left[ ( p-1 )
  \frac{f_{\sigma}^{p-2}}{\mathring{\alpha}^{1- \sigma}}
  \mathring{h}^{\alpha}_{i j}   \nabla_{i} f_{\sigma}   \nabla_{j} H^{\alpha}
  \right. \nonumber \\
  &  & \left. - ( 1- \sigma ) \frac{f_{\sigma}^{p-1}}{\mathring{\alpha}^{2-
  \sigma}}   \mathring{h}^{\alpha}_{i j}   \nabla_{i} \mathring{\alpha}
  \nabla_{j} H^{\alpha} + \frac{n-1}{n}
  \frac{f_{\sigma}^{p-1}}{\mathring{\alpha}^{1- \sigma}} | \nabla H |^{2}
  \right] \mathd \mu_{t}\\
  & \le & \int_{M_{t}} \left[ ( p-1 )
  \frac{f_{\sigma}^{p-2}}{\mathring{\alpha}^{1- \sigma}} \normho | \nabla
  f_{\sigma} | | \nabla H | \right. \nonumber \\
  &  & \left. + \frac{f_{\sigma}^{p-1}}{\mathring{\alpha}^{2- \sigma}}
  \normho \left| \nabla \mathring{\alpha} \right| |   \nabla H | +
  \frac{f_{\sigma}^{p-1}}{\mathring{\alpha}^{1- \sigma}} | \nabla H |^{2}
  \right] \mathd \mu_{t} \nonumber \\
  & \le & \int_{M_{t}} \left[ ( p-1 ) \frac{f_{\sigma}^{p-1}}{\normho}
  | \nabla f_{\sigma} | |   \nabla H | + \frac{2f_{\sigma}^{p}}{\normho^{2}}
  |   \nabla H |^{2} \right] \mathd \mu_{t} \nonumber
\end{eqnarray}
and
\begin{eqnarray}
  \int_{M_{t}} \frac{f_{\sigma}^{p}}{\mathring{\alpha}} \Delta
  \mathring{\alpha}   \mathd \mu_{t} & = & - \int_{M_{t}} \left\langle \nabla
  \left( \frac{f_{\sigma}^{p}}{\mathring{\alpha}} \right) , \nabla
  \mathring{\alpha} \right\rangle \mathd \mu_{t} \nonumber \\
  & = & \int_{M_{t}} \left( - \frac{p f_{\sigma}^{p-1}}{\mathring{\alpha}}
  \left\langle \nabla f_{\sigma} , \nabla \mathring{\alpha} \right\rangle +
  \frac{f_{\sigma}^{p}}{\mathring{\alpha}^{2}} \left| \nabla \mathring{\alpha}
  \right|^{2} \right) \mathd \mu_{t}\label{inte2}\\
  & \le & \int_{M_{t}} \left( \frac{p f_{\sigma}^{p-1}}{\normho} |
  \nabla f_{\sigma} | |   \nabla H | + \frac{f_{\sigma}^{p}}{\normho^{2}} |
  \nabla H |^{2} \right) \mathd \mu_{t} . \nonumber
\end{eqnarray}
Putting (\ref{inte1})-(\ref{inte2}) together, we obtain
\begin{equation}
  \frac{\varepsilon}{2} \int_{M_{t}} | H |^{2} f_{\sigma}^{p} \mathd \mu_{t}
  \le \int_{M_{t}} \left( \frac{3p f_{\sigma}^{p-1}}{\normho}   | \nabla
  f_{\sigma} | |   \nabla H |  +  \frac{5f_{\sigma}^{p}}{\normho^{2}}   |
  \nabla H |^{2} \right)   \mathd \mu_{t} . \label{H2fsp}
\end{equation}
From (\ref{H2fsp}) and Lemma \ref{ptf}, we get an estimate for the
time derivative of the integral of $f_{\sigma}^{p}$.
\begin{eqnarray}
  \frac{\mathd}{\mathd t} \int_{M_{t}} f_{\sigma}^{p} \mathd \mu_{t} & = & p
  \int_{M_{t}} f_{\sigma}^{p-1} \frac{\partial f_{\sigma}}{\partial t} \mathd
  \mu_{t} - \int_{M_{t}} f_{\sigma}^{p} | H |^{2} \mathd \mu_{t} \nonumber\\
  & \le & p \int_{M_{t}} \left[ f_{\sigma}^{p-1} \Delta f_{\sigma}  +
  \frac{2 f_{\sigma}^{p-1}}{\normho} | \nabla f_{\sigma} | | \nabla H |
  \right. \nonumber\\
  &  & \left. - \frac{12  \varepsilon  f_{\sigma}^{p}}{\normho^{2}} | \nabla
  H |^{2} -4 c f_{\sigma}^{p} +  \sigma | H |^{2} f_{\sigma}^{p} \right]
  \mathd \mu_{t} \nonumber\\
  & \le & p \int_{M_{t}} \left[ - ( p-1 ) f_{\sigma}^{p-2} | \nabla
  f_{\sigma} |^{2} + \frac{2 f_{\sigma}^{p-1}}{\normho} | \nabla f_{\sigma} |
  | \nabla H | - \frac{12  \varepsilon  f_{\sigma}^{p}}{\normho^{2}} | \nabla
  H |^{2} \right.  \label{dtint}\\
  &  & \left. + \frac{6  \sigma  p}{\varepsilon}
  \frac{f_{\sigma}^{p-1}}{\normho} | \nabla f_{\sigma} | |   \nabla H | +
  \frac{10  \sigma}{\varepsilon}   \frac{f_{\sigma}^{p}}{\normho^{2}}   |
  \nabla H |^{2} -4 c f_{\sigma}^{p} \right] \mathd \mu_{t} \nonumber\\
  & = & p \int_{M_{t}} f_{\sigma}^{p-2} \left[ - ( p-1 ) | \nabla f_{\sigma}
  |^{2} + \left( 2+ \frac{6  \sigma  p}{\varepsilon} \right)
  \frac{f_{\sigma}}{\normho} | \nabla f_{\sigma} | | \nabla H | \right.
  \nonumber\\
  &  & \left. - \left( 12  \varepsilon - \frac{10  \sigma}{\varepsilon}
  \right) \frac{f_{\sigma}^{2}}{\normho^{2}} | \nabla H |^{2} \right] \mathd
  \mu_{t} -4 c p \int_{M_{t}} f_{\sigma}^{p} \mathd \mu_{t} . \nonumber
\end{eqnarray}

Now we show that the $L^{p}$-norm of $f_{\sigma}$ is bounded.

\begin{lemma}
  \label{pnorm}There exists a constant $C_{1}$ depending only on $M_{0}$ such
  that for all $p \ge 1/ \varepsilon$ and $\sigma \le
  \varepsilon^{2} / \sqrt{p}$, we have
  \[ \left( \int_{M_{t}} f_{\sigma}^{p} \mathd \mu_{t} \right)^{\frac{1}{p}}
     <C_{1} . \]
\end{lemma}

\begin{proof}
  The expression in the square bracket of the right hand side of (\ref{dtint})
  is a quadratic form. With $\varepsilon$ small enough, its discriminant
  satisfies
  \begin{eqnarray*}
    &  & \left( 2+ \frac{6  \sigma  p}{\varepsilon} \right)^{2} -4 ( p-1 )
    \left( 12  \varepsilon - \frac{10  \sigma}{\varepsilon} \right)\\
    & < & \left( 2+6  \sqrt{p}   \varepsilon \right)^{2} -24p  \varepsilon\\
    & \le & 8+72p  \varepsilon^{2} -24p  \varepsilon\\
    & < & 0.
  \end{eqnarray*}
  Therefore, this quadratic form is non-positive. Now we have
  \[ \frac{\mathd}{\mathd t} \int_{M_{t}} f_{\sigma}^{p} \mathd \mu_{t}
     \le -4 c p \int_{M_{t}} f_{\sigma}^{p} \mathd \mu_{t} . \]
  This implies $\int_{M_{t}} f_{\sigma}^{p} \mathd \mu_{t} \le e^{-4c p
  t} \int_{M_{0}} f_{\sigma}^{p} \mathd \mu_{0}$. Thus, the assertion follows
  from the finiteness of $T$.
\end{proof}

\begin{corollary}
  \label{pnorm2}There exists a constant $C_{2}$ depending only on $M_{0}$ such
  that for all $r \ge 1$, $p \ge 4 r^{2}  / \varepsilon^{4}$ and
  $\sigma \le   \frac{1}{2}   \varepsilon^{2}  / \sqrt{p}$, we have
  \[ \left( \int_{M_{t}} | H |^{2r}  f_{\sigma}^{p} \mathd \mu_{t}
     \right)^{\frac{1}{p}} <C_{2} . \]
\end{corollary}

\begin{proof}
  Since the constant $C ( \delta ,n ) = \sup \left\{ \frac{| H
  |^{2}}{\mathring{\alpha}}   \,\middle|\,   | H |^{2} >-n^{2} c+ \delta \right\}$
  is finite, we have
  \[ \left( \int_{M_{t}} | H |^{2r}  f_{\sigma}^{p} \mathd \mu_{t}
     \right)^{\frac{1}{p}} \le \left( \int_{M_{t}} \left( C ( \delta ,n
     ) \mathring{\alpha} \right)^{r}  f_{\sigma}^{p} \mathd \mu_{t}
     \right)^{\frac{1}{p}} \le C ( \delta ,n )^{\frac{r}{p}} \left(
     \int_{M_{t}}  f_{\sigma + \frac{r}{p}}^{p}   \mathd \mu_{t}
     \right)^{\frac{1}{p}} . \]
  With $r/p \le \frac{1}{2} \varepsilon^{2}   / \sqrt{p}$ and $\sigma
  +r/p \le   \varepsilon^{2}   / \sqrt{p}$, the conclusion follows from
  Lemma \ref{pnorm}.
\end{proof}

We are now in a position to complete the proof of Theorem 4.1.
\\ \textit{Proof of Theorem 4.1.} For all $k>0$, we define $f_{\sigma
,k} = \max \{ f_{\sigma} -k,0 \}$, $A ( k ) = \{ x \in M_{t}   \,|\,
f_{\sigma} ( x ) >k \}$. From Lemma \ref{ptf} and $p \ge 1/
\varepsilon$, we have
\begin{eqnarray*}
  \dt \int_{M_{t}} f_{\sigma ,k}^{p} \mathd \mu_{t} & \le & p
  \int_{M_{t}} f_{\sigma ,k}^{p-1} \left( \Delta f_{\sigma} + \tfrac{2 |
  \nabla f_{\sigma} | | \nabla H |}{\normho} - \tfrac{12  \varepsilon
  f_{\sigma} | \nabla H |^{2}}{\normho^{2}} -4 c f_{\sigma} + \sigma | H |^{2}
  f_{\sigma} \right) \mathd \mu_{t}\\
  & \le & p \int_{M_{t}}   \left[ - ( p-1 ) f_{\sigma ,k}^{p-2} |
  \nabla f_{\sigma} |^{2} + \tfrac{2 f_{\sigma ,k}^{p-1} | \nabla f_{\sigma} |
  | \nabla H |}{\normho} - \tfrac{12  \varepsilon f_{\sigma ,k}^{p} | \nabla H
  |^{2}}{\normho^{2}} \right] \mathd \mu_{t}\\
  &  & + p \int_{A ( k )} (  -4 c+ \sigma   | H |^{2} ) f_{\sigma}^{p} \mathd
  \mu_{t}\\
  & \le & - \frac{1}{2} p ( p-1 ) \int_{M_{t}} f_{\sigma ,k}^{p-2} |
  \nabla f_{\sigma} |^{2} \mathd \mu_{t} + p \int_{A ( k )} ( \sigma | H |^{2}
  -4 c ) f_{\sigma}^{p} \mathd \mu_{t} .
\end{eqnarray*}
We have $\frac{1}{2} p ( p-1 ) f_{\sigma ,k}^{p-2} | \nabla
f_{\sigma} |^{2} \ge | \nabla f_{\sigma ,k}^{p/2} |^{2}$. Putting
$v=f_{\sigma ,k}^{p/2}$, we get
\begin{equation}
  \dt \int_{M_{t}} v^{2} \mathd \mu_{t} + \int_{M_{t}} | \nabla v |^{2} \mathd
  \mu_{t} \le  p \int_{A ( k )} ( \sigma | H |^{2} -4 c ) f_{\sigma}^{p}
  \mathd \mu_{t} . \label{ineq1}
\end{equation}

From Theorem 2.1 of {\cite{MR0365424}}, we see that if $u$ is a
non-negative $C^{1}$-function on $M_{t}$, then the following Sobolev
inequality holds: $( \int_{M_{t}} u^{\frac{n}{n-1}} \mathd \mu_{t}
)^{\frac{n-1}{n}} \le C_{3} \int_{M_{t}} ( | \nabla u | +u | H | )
\mathd \mu_{t}$, where $C_{3}$ is a positive constant depending only
on $n$. Replacing $u$ by $v^{2 ( n-1 ) / ( n-2 )}$ and using
H{\"o}lder's inequality, we obtain
\[ \left( \int_{M_{t}} v^{\frac{2n}{n-2}} \mathd \mu_{t}
   \right)^{\frac{n-2}{n}} \le C_{3} \int_{M_{t}} | \nabla v |^{2}
   \mathd \mu_{t} +C_{3} \left( \int_{A ( k )} | H |^{n} \mathd \mu_{t}
   \right)^{\frac{2}{n}} \left( \int_{M_{t}} v^{\frac{2n}{n-2}} \mathd \mu_{t}
   \right)^{\frac{n-2}{n}} . \]
When $k \ge C_{2} ( 2 C_{3} )^{\frac{n}{2p}}$, $p \ge n^{2} /
\varepsilon^{4}$ and $\sigma \le   \frac{1}{2}   \varepsilon^{2} /
\sqrt{p}$, it follows from Corollary \ref{pnorm2} that $( \int_{A (
k )} | H |^{n} \mathd \mu_{t} )^{\frac{2}{n}} \le ( \int_{A ( k )} |
H |^{n} f_{\sigma}^{p} k^{-p} \mathd \mu_{t} )^{\frac{2}{n}} <
\frac{1}{2 C_{3}}$. Thus, we have
\begin{equation}
  \frac{1}{2 C_{3}} \left( \int_{M_{t}} v^{\frac{2n}{n-2}} \mathd \mu_{t}
  \right)^{\frac{n-2}{n}} \le \int_{M_{t}} | \nabla v |^{2} \mathd
  \mu_{t} . \label{vlessdv}
\end{equation}

It follows from (\ref{ineq1}) and (\ref{vlessdv}) that
\[ \dt \int_{M_{t}} v^{2} \mathd \mu_{t} + \frac{1}{2 C_{3}} \left(
   \int_{M_{t}} v^{\frac{2n}{n-2}} \mathd \mu_{t} \right)^{\frac{n-2}{n}}
   \le  p \int_{A ( k )} ( \sigma | H |^{2} -4 c ) f_{\sigma}^{p} \mathd
   \mu_{t} . \]
Letting $k \ge \sup_{M_{0}} f_{\sigma}$, we have $v=0$ at $t=0$.
Integrating the both sides of the above inequality, we get
\begin{eqnarray}
  \underset{[ 0,T ]}{\sup} \int_{M_{t}} v^{2} \mathd \mu_{t} & + & \frac{1}{2
  C_{3}} \int_{0}^{T} \left( \int_{M_{t}} v^{\frac{2n}{n-2}} \mathd \mu_{t}
  \right)^{\frac{n-2}{n}} \mathd t \nonumber\\
  & \le & 2 p \int_{0}^{T} \int_{A ( k )} ( \sigma | H |^{2} -4 c )
  f_{\sigma}^{p} \mathd \mu_{t}   \mathd t.  \label{firstof5}
\end{eqnarray}
Let $\| A ( k ) \| = \int_{0}^{T} \int_{A ( k )} \mathd \mu_{t}
\mathd t$. For a positive number $r> \frac{n+2}{2}$, let $p \ge 4r/
\varepsilon^{4}$, $\sigma \le \frac{1}{2} \varepsilon^{2} / \sqrt{p
r}$. Applying H{\"o}lder's inequality and Lemma \ref{pnorm}, we have
\begin{equation}
  \int_{0}^{T} \int_{A ( k )} f_{\sigma}^{p} \mathd \mu_{t}   \mathd t
  \le C_{1}^{p}  T^{\frac{1}{r}}   \| A ( k ) \|^{1- \frac{1}{r}} .
\end{equation}
By H{\"o}lder's inequality and Corollary \ref{pnorm2}, we have
\begin{equation}
  \int_{0}^{T} \int_{A ( k )} | H |^{2} f_{\sigma}^{p} \mathd \mu_{t}   \mathd
  t \le C_{2}^{p}  T^{\frac{1}{r}}   \| A ( k ) \|^{1- \frac{1}{r}} .
\end{equation}
For $h>k$, we have $f_{\sigma ,k} >h-k$ on $A ( h )$. Thus
\begin{equation}
  ( h-k )^{p} \| A ( h ) \| \le \int_{0}^{T} \int_{M_{t}} v^{2} \mathd
  \mu_{t}   \mathd t \le \left( \int_{0}^{T} \int_{M_{t}} v^{\frac{2 (
  n+2 )}{n}} \mathd \mu_{t}   \mathd t \right)^{\frac{n}{n+2}} \| A ( k )
  \|^{\frac{2}{n+2}} .
\end{equation}
We estimate the right hand side of the above inequality as follows.
\begin{eqnarray}
  &  & \left( \int_{0}^{T} \int_{M_{t}} v^{\frac{2 ( n+2 )}{n}} \mathd
  \mu_{t}   \mathd t \right)^{\frac{n}{n+2}} \nonumber\\
  & \le & \left[ \int_{0}^{T} \left( \int_{M_{t}} v^{2} \mathd \mu_{t}
  \right)^{\frac{2}{n}} \left( \int_{M_{t}} v^{\frac{2n}{n-2}} \mathd \mu_{t}
  \right)^{\frac{n-2}{n}} \mathd t \right]^{\frac{n}{n+2}} \nonumber\\
  & \le & \left( \underset{[ 0,T ]}{\sup} \int_{M_{t}} v^{2} \mathd
  \mu_{t} \right)^{\frac{2}{n+2}} \left[ \int_{0}^{T} \left( \int_{M_{t}}
  v^{\frac{2n}{n-2}} \mathd \mu_{t} \right)^{\frac{n-2}{n}} \mathd t
  \right]^{\frac{n}{n+2}}  \label{fifthof5}\\
  & \le & \frac{2}{n+2}   \underset{[ 0,T ]}{\sup} \int_{M_{t}} v^{2}
  \mathd \mu_{t} + \frac{n}{n+2}   \int_{0}^{T} \left( \int_{M_{t}}
  v^{\frac{2n}{n-2}} \mathd \mu_{t} \right)^{\frac{n-2}{n}} \mathd t
  \nonumber\\
  & \le & C_{4} \left[ \underset{[ 0,T ]}{\sup} \int_{M_{t}} v^{2}
  \mathd \mu_{t} + \frac{1}{2 C_{3}}   \int_{0}^{T} \left( \int_{M_{t}}
  v^{\frac{2n}{n-2}} \mathd \mu_{t} \right)^{\frac{n-2}{n}} \mathd t \right] .
  \nonumber
\end{eqnarray}
Putting inequalities (\ref{firstof5})-(\ref{fifthof5}) together, for all $h>k
\ge \max \left\{ C_{2} ( 2 C_{3} )^{\frac{n}{2p}} , \sup_{M_{0}}
f_{\sigma} \right\}$, we have
\[ ( h-k )^{p} \| A ( h ) \| \le   C_{4}   \| A ( k ) \|^{1- \frac{1}{r}
   + \frac{2}{n+2}} , \]
where $C_{4}$ is a positive constant depending on $M_{0}$, $p$ and $r$.

By a lemma of {\cite{MR1786735}} (Chapter II, Lemma B.1), there
exists a finite number $k_{1}$, such that $\| A ( k_{1} ) \| =0$.
Therefore, the assertion follows from the definition of $A ( k )$
and Lemma \ref{app} (iii). This completes the proof of Theorem
\ref{sa0h2}.
\hfill \qedsymbol

\section{A gradient estimate}

To compare the mean curvature at different points of $M_{t}$, we
need derive an estimate for the gradient of mean curvature.

\begin{theorem}
  \label{dH2}For all $\eta \in \left( 0, \frac{1}{n} \right)$, there exists a
  constant $C ( \eta )$ independent of t, such that
  \[ | \nabla H | < \eta^{2}   | H |^{2} +C ( \eta ) . \]
\end{theorem}

Firstly, we make an estimate for the time derivative of the gradient
of mean curvature.

\begin{lemma}
  \label{dtdH2}There exists a constant $B_{1} >1$ depending only on $n$, such
  that
  \[ \dt | \nabla H |^{2} \le \Delta | \nabla H |^{2} +B_{1} | H |^{2} |
     \nabla h |^{2} . \]
\end{lemma}

\begin{proof}

  The evolution equation of $H$ is
  \begin{equation}
    \nabla_{\partial_{t}} H= \Delta H+n c H+H^{\alpha} h^{\alpha}_{i j} h_{i
    j} .
  \end{equation}
  Let $u$ be a tangent vector field on $M_{t}$ satisfying $[ \partial_{t} ,u ]
  =0$. Then we have $\nabla_{\partial_{t}} u= ( \bar{\nabla}_{\partial_{t}} u
  )^{\top} = ( \bar{\nabla}_{u} \partial_{t} )^{\top} =- \langle H,h ( u,e_{i}
  ) \rangle e_{i}$.

  Using the timelike Ricci equation (see (16) of {\cite{MR2739807}}), we
  obtain
  \[ \nabla_{\partial_{t}} ( \nabla_{u} H ) = \nabla_{u} (
     \nabla_{\partial_{t}} H ) + \langle H,h ( u,e_{i} ) \rangle \nabla_{i} H-
     \langle H, \nabla_{i} H \rangle h ( u,e_{i} ) . \]
  Therefore, we have
  \begin{eqnarray*}
    \nabla_{\partial_{t}} \nabla H ( u ) & = & \nabla_{\partial_{t}} (
    \nabla_{u} H ) - \nabla H ( \nabla_{\partial_{t}} u )\\
    & = & \nabla_{u} ( \Delta H+n c H+H^{\alpha} h^{\alpha}_{i j} h_{i j} )\\
    &  & +2 \langle H,h ( u,e_{i} ) \rangle \nabla_{i} H- \langle H,
    \nabla_{i} H \rangle h ( u,e_{i} ) .
  \end{eqnarray*}
  We use Hamilton's $\ast$ notation. For tensors $T$ and $S$, $T \ast S$ means
  any linear combination of contractions of $T$ and $S$ with the metric. Then
  the above formula can be written as
  \[ \nabla_{\partial_{t}} \nabla H= \nabla \Delta H+n c \nabla H+h \ast h
     \ast \nabla h. \]
  It follows from Ricci equation and Gauss equation that
  \[ \nabla \Delta H= \Delta \nabla H+ ( 1-n ) c  \nabla H+h \ast h \ast
     \nabla h. \]
  Now we obtain the evolution equation of $\nabla H$
  \begin{equation}
    \nabla_{\partial_{t}} \nabla H= \Delta \nabla H+c  \nabla H+h \ast h \ast
    \nabla h.
  \end{equation}
  Then
  \begin{eqnarray}
    \dt | \nabla H |^{2} & = & 2 \langle \nabla_{\partial_{t}} \nabla H,
    \nabla H \rangle \nonumber\\
    & = & 2 \langle \Delta \nabla H, \nabla H \rangle +2c | \nabla H |^{2}
    +h \ast h \ast \nabla h \ast \nabla h \\
    & = & \Delta | \nabla H |^{2} -2  | \nabla^{2} H |^{2} +2c | \nabla H
    |^{2}  +h \ast h \ast \nabla h \ast \nabla h. \nonumber
  \end{eqnarray}

  From the Cauchy-Schwarz inequality and the pinching condition, we have $| h
  \ast h \ast \nabla h \ast \nabla h | \le B_{1} | H |^{2} | \nabla h
  |^{2}$.
\end{proof}

Secondly, we need the following estimates.

\begin{lemma}
  \label{threedtlap}Along the mean curvature flow, we have
  \begin{enumerateroman}
    \item $\dt | H |^{4} \ge \Delta | H |^{4} -8 n  | H |^{2} | \nabla h
    |^{2} + \frac{4}{n} | H |^{6} +4 n c | H |^{4}$,

    \item $\dt \normho^{2} \le \Delta \normho^{2} - | \nabla h |^{2} + |
    H |^{4}$,

    \item $\dt \left( | H |^{2} \normho^{2} \right) \le \Delta \left( |
    H |^{2} \normho^{2} \right) - \frac{1}{2}   | H |^{2} | \nabla h |^{2}
    +B_{2} | \nabla h |^{2} +C_{0}   | H |^{6-2 \sigma}$, where $B_{2}$ is a
    positive constant.
  \end{enumerateroman}
\end{lemma}

\begin{proof}

  (i) From the evolution equation we derive that
  \[ \dt | H |^{4} = \Delta | H |^{4} -4 | H |^{2} | \nabla H |^{2} -2 |
     \nabla | H |^{2} |^{2} +4 | H |^{2} R_{2} +4 n c  | H |^{4} . \]
  Then from Lemma \ref{dA2} and \ref{R12}, we have $R_{2} > \frac{1}{n} | H
  |^{4}$ and $4 | H |^{2} | \nabla H |^{2} +2 | \nabla | H |^{2} |^{2}
  \le 8 n  | H |^{2} | \nabla h |^{2}$.

  (ii) The evolution equation of $\normho^{2}$ is
  \[ \dt \normho^{2} = \Delta \normho^{2} -2 | \nabla h |^{2} + \frac{2}{n} |
     \nabla H |^{2} +2R_{1} - \frac{2}{n} R_{2} -2 n c \normho^{2} . \]
  Then from $\frac{2}{n} | \nabla H |^{2} \le | \nabla h |^{2}$ and
  $R_{1} - \frac{1}{n} R_{2} -n c \normho^{2} \le \normho^{2} ( | h
  |^{2} -n c ) < \frac{1}{2} | H |^{4}$ we get the conclusion.

  (iii) It follows from the evolution equations that
  \begin{eqnarray*}
    \dt \left( | H |^{2} \normho^{2} \right) & = & \Delta \left( | H |^{2}
    \normho^{2} \right) +2 | H |^{2} \left( R_{1} - \frac{1}{n} R_{2} \right)
    +2 \normho^{2} R_{2}\\
    &  & -2  | H |^{2} \left( | \nabla h |^{2} - \frac{1}{n} | \nabla H |^{2}
    \right) -2 \normho^{2} | \nabla H |^{2} -2 \left\langle \nabla | H |^{2} ,
    \nabla \normho^{2} \right\rangle .
  \end{eqnarray*}
  We use $\frac{2}{n} | \nabla H |^{2} \le | \nabla h |^{2}$ again. Then
  by Theorem \ref{sa0h2} we have
  \[ 2 | H |^{2} \left( R_{1} - \frac{1}{n} R_{2} \right) +2 \normho^{2} R_{2}
     \le 4  | H |^{2} | h |^{2} \normho^{2} <C_{0} | H |^{6-2 \sigma} .
  \]
  From the formula $\nabla_{i} \normho^{2} =2 \mathring{h}^{\alpha}_{j k}
  \nabla_{i} h^{\alpha}_{j k}$ and Young's inequality, we get
  \begin{eqnarray*}
    -2 \left\langle \nabla | H |^{2} , \nabla \normho^{2} \right\rangle &
    \le & 8  | H |   | \nabla H |   \normho   | \nabla h |\\
    & \le & 8 \sqrt{\tfrac{n+2}{3} C_{0}}   | H |^{2- \sigma} | \nabla
    h |^{2}\\
    & \le & \left( B_{2} + \frac{1}{2}   | H |^{2} \right) | \nabla h
    |^{2} .
  \end{eqnarray*}
This proves Lemma 5.3.
\end{proof}

Now we can prove Theorem \ref{dH2}.
\\ \textit{Proof of Theorem \ref{dH2}.} We define a function on $M$.
\[ f= | \nabla H |^{2} - \eta^{4}   | H |^{4} +4B_{1}   | H |^{2} \normho^{2}
   +4B_{1} B_{2}   \normho^{2} , \hspace{1em} 0< \eta < \frac{1}{n} . \]
From Lemmas \ref{dtdH2} and \ref{threedtlap}, we obtain
\begin{eqnarray*}
  \left( \dt - \Delta \right) f & \le & B_{1} | H |^{2} | \nabla h |^{2}
  - \eta^{4} \left( -8 n  | H |^{2} | \nabla h |^{2} + \frac{4}{n} | H |^{6}
  +4 n c | H |^{4} \right)\\
  &  & +4B_{1} \left( - \frac{1}{2}   | H |^{2} | \nabla h |^{2} +B_{2} |
  \nabla h |^{2} +C_{0}   | H |^{6-2 \sigma} \right) \\
    &  &+4B_{1} B_{2} ( - |
  \nabla h |^{2} + | H |^{4} )\\
  & \le & - \eta^{4}   \left( \frac{4}{n} | H |^{6} +4 n c | H |^{4}
  \right) + 4B_{1}  C_{0}   | H |^{6-2 \sigma}  +4 B_{1} B_{2} | H |^{4} .
\end{eqnarray*}
We consider the last line
of the above inequality, which is a function of $| H |$. Since the
coefficient of the highest-degree term is negative, the supremum
$C_{2} ( \eta )$ of this function is finite. Then we have
\[ \dt f< \Delta f+C_{2} ( \eta ) . \]
It's seen from the maximum principle that $f$ is bounded. This
completes the proof of Theorem \ref{dH2}.
\hfill \qedsymbol

\section{Convergence}

To show that $M_{t}$ converges to a point, we derive a lower bound for the
Ricci curvature.

\begin{lemma}
  \label{ric}If $M$ is a submanifold in $\mathbb{H}^{n+q} ( c )$
  satisfying
  $\normho^{2} < \mathring{\alpha} - \varepsilon   \omega$ and $| H |^{2}
  +n^{2} c>0$, then for all unit vector $X$ in the tangent space, the Ricci
  curvature satisfies
  \[ \tmop{Ric} ( X ) \ge \frac{n-1}{4 n} \varepsilon | H |^{2} . \]
\end{lemma}

\begin{proof}
  Using Proposition 2 in {\cite{MR1458750}} and Lemma \ref{app} (ii), we have
  \begin{eqnarray*}
    \tmop{Ric} ( X ) & \ge & \frac{n-1}{n} \left( n c+ \frac{2}{n} | H
    |^{2} - | h |^{2} - \frac{n-2}{\sqrt{n ( n-1 )}} | H | \normho \right)\\
    & > & \frac{n-1}{n} \left( n c+ \frac{1}{n} | H |^{2} - \left(
    \mathring{\alpha} - \varepsilon   \omega \right) - \frac{n-2}{\sqrt{n (
    n-1 )}} | H | \sqrt{\mathring{\alpha}} \right)\\
    & = & \frac{n-1}{n}   \varepsilon   \omega\\
    & > & \frac{n-1}{4 n}   \varepsilon   | H |^{2} .
  \end{eqnarray*}

\end{proof}

To estimate the diameter of $M_{t}$, we need the well-known
Myers theorem.

\begin{theorem}
  [\textbf{Myers}] Let $\gamma$ be a geodesic of length at least $\pi /
  \sqrt{k}$ on $M$. If the Ricci curvature satisfies $\tmop{Ric} ( X )
  \ge ( n-1 ) k$ for all $x \in \gamma$, all unit vector $X \in T_{x}
  M$, then $\gamma$ has conjugate points.
\end{theorem}

Let $| H |_{\min} = \min_{M_{t}} | H |$, $| H |_{\max} =
\max_{M_{t}} | H |$. To show the flow converges to a round point, we
need the following lemma.

\begin{lemma}
  As $t \rightarrow T$, we have $\tmop{diam}  M_{t} \rightarrow 0$ and $| H
  |_{\min} / | H |_{\max} \rightarrow 1$.
\end{lemma}

\begin{proof}
  Theorem \ref{dH2} asserts $| \nabla H | < \eta^{2} | H |^{2} +C ( \eta )$
  for $0< \eta < \frac{1}{n}$. Since $| H |_{\max} \rightarrow \infty$ as $t
  \rightarrow T$, there exists a time $\tau ( \eta )$, such that for $t> \tau
  ( \eta )$, $| H |_{\max}^{2} >C ( \eta ) / \eta^{2}$. Then we have $| \nabla
  H | <2  \eta^{2} | H |^{2}_{\max}$.

  At a time $t> \tau ( \eta )$, let $x$ be a point on $M_{t}$ where $| H |$
  achieves its maximum. Then along all geodesics of length $l= ( 2  \eta   | H
  |_{\max} )^{-1}$ starting from $x$, we have $| H | > | H |_{\max} - | \nabla
  H | \cdot l> ( 1- \eta ) | H |_{\max}$. With $\eta$ small enough, Lemma
  \ref{ric} implies $\tmop{Ric} > \frac{n-1}{4 n} \varepsilon ( 1- \eta )^{2}
  | H |^{2}_{\max} > ( n-1 ) \pi^{2} /l^{2}$ on these geodesics. Then from
  Myers' theorem, these geodesics can reach any point of $M_{t}$.

  Thus we have $| H |_{\min} > ( 1- \eta ) | H |_{\max}$ and $\tmop{diam}
  M_{t} \le ( 2  \eta   | H |_{\max} )^{-1}$ for $t \in ( \tau ( \eta )
  ,T )$. This proves the lemma.
\end{proof}

Now we are in a position to complete the proof of Main Theorem.\\
\textit{Proof of Main Theorem.} To prove the flow converges to a round point, we magnify the metric of the
ambient space such that the submanifold maintains its volume along
the flow. Using the same argument as in {\cite{liu2012mean}}, we can
prove that the rescaled mean curvature flow converges to a totally
umbilical sphere as the reparameterized time tends to infinity. This
completes the proof of Main Theorem.
\hfill \qedsymbol

\section{Final Remarks}

Since $\mathring{\alpha} \rightarrow 0$ as $| H |^{2} \rightarrow
-n^{2} c$, we see that if a connected submanifold satisfies $\sup
\left( \normho^{2} - \mathring{\alpha} \right) <0$, then either $| H
|^{2}
>-n^{2} c$ or $| H |^{2} <-n^{2} c$. Therefore, the condition in Main Theorem can be relaxed.

\begin{theorem}
  Let $F_{0} :M^{n} \rightarrow \mathbb{H}^{n+q} ( c )$ be an n-dimensional
  ($n \ge 6$) complete submanifold immersed in the hyperbolic space with
  constant curvature $c$. Suppose
  \[ \sup_{F_{0}} ( | h |^{2} - \alpha ( n, | H | ,c ) ) <0 \]
  and there exists a point $x$ on $F_{0}$ such that $| H ( x ) |^{2} >-n^{2}
  c$. Then the mean curvature flow with initial value $F_{0}$ has a unique
  smooth solution $F: M \times [ 0,T ) \rightarrow \mathbb{H}^{n+q} ( c )$ on
  a finite maximal time interval, and $F_{t}$ converges to a round point as $t
  \rightarrow T$. In particular, $M$ is diffeomorphic to the standard
  $n$-sphere $\mathbb{S}^{n}$.
\end{theorem}

For the compact submanifolds, we have the following convergence
theorem under the weakly pinching condition.

\begin{theorem}
  Let $F_{0} :M^{n} \rightarrow \mathbb{H}^{n+q} ( c )$ be an n-dimensional
  ($n \ge 6$) closed submanifold immersed in the hyperbolic space with
  constant curvature $c$. If $F_{0}$ satisfies
  \[ | h |^{2} \le \alpha ( n, | H | ,c ) \hspace{1em} \tmop{and}
     \hspace{1em} | H |^{2} +n^{2} c>0, \]
  then the mean curvature flow with initial value $F_{0}$ has a unique smooth
  solution $F: M \times [ 0,T ) \rightarrow \mathbb{H}^{n+q} ( c )$ on a
  finite maximal time interval, and $F_{t}$ converges to a round point as $t
  \rightarrow T$. In particular, $M$ is diffeomorphic to the standard
  $n$-sphere $\mathbb{S}^{n}$.
\end{theorem}

\begin{proof}
  By the continuity, there exists $t_{0} >0$ such that $| H |^{2} +n^{2} c>0$
  remains true for $t \in [ 0,t_{0} ]$. Similar to the proof of Theorem
  \ref{pinch}, we obtain the following estimate on the time interval $[
  0,t_{0} ]$.
  \begin{eqnarray}
    &  & \left( \dt - \Delta \right) \left( \normho^{2} - \mathring{\alpha}
    \right) \nonumber\\
    & \le & \frac{2 ( n+2 )}{3} \left[ - \frac{2 ( n-1 )}{n ( n+2 )} +
    \mathring{\alpha}' +2 | H |^{2} \mathring{\alpha}''   \right] | \nabla h
    |^{2} \nonumber\\
    &  & +2P_{2} \left[ 2 \mathring{\alpha} - \frac{1}{n} | H |^{2} + | H
    |^{2} \mathring{\alpha}' \right]  \label{Phosao}\\
    &  & +2 \left( \normho^{2} - \mathring{\alpha} \right) \left( 2
    \mathring{\alpha} + \frac{1}{n} | H |^{2} -n c - \mathring{\alpha}' | H
    |^{2} +2P_{2} \right) +2 \left( \normho^{2} - \mathring{\alpha}
    \right)^{2} . \nonumber
  \end{eqnarray}
  The expressions in the two square brackets of (\ref{Phosao}) are negative.
  Thus the weak maximum principle implies $\normho^{2} \le
  \mathring{\alpha}$ is preserved for $t \in [ 0,t_{0} ]$.

  Then we use the strong maximum principle. If $\sup_{M_{t_{0}}} \left(
  \normho^{2} - \mathring{\alpha} \right) =0$, then $\normho^{2} \equiv
  \mathring{\alpha}$ holds for all $t \in [ 0,t_{0} ]$. So, we have $| \nabla
  h |^{2} \equiv P_{2} \equiv 0$ for $t \in [ 0,t_{0} ]$. By a theorem due to Erbacher
  {\cite{MR0288701}}, $M_{0}$ lies in an $( n+1 )$-dimensional totally
  geodesic submanifold of $\mathbb{H}^{n+q} ( c )$. It follows from Theorem 4 of
  {\cite{MR0238229}} and compactness of $M_{0}$ that $M_{0}$ is a totally
  umbilical sphere. This contradicts \ $\normho^{2} \equiv \mathring{\alpha}
  >0$.

  Therefore, we have $\sup_{M_{t_{0}}} \left( \normho^{2} - \mathring{\alpha}
  \right) <0$. The assertion follows from Main Theorem.
\end{proof}

Finally we discuss the convergence and sphere theorems in low
dimensions. In the case where $n\leq5$, $\mathring{\alpha}$ does not
possesses all the properties in Lemma \ref{app}, which prevents us
from getting the same convergence result as Main Theorem.

When $n=5$, we obtain a convergence result for the mean curvature
flow, which improves the corresponding result in the convergence
theorem due to Liu-Xu-Ye-Zhao {\cite{MR3078951}}.

Set
\[ \beta ( x ) = \tfrac{5}{11} c+ \tfrac{15}{88} x+ \tfrac{\sqrt{7}}{88}
   \sqrt{7x^{2} +272c x+4000c^{2}} . \]
Then we obtain the following convergence result.

\begin{proposition}
  \label{fived}Let $F_{0} :M^{5} \rightarrow \mathbb{H}^{5+q} ( c )$ be a
  complete submanifold immersed in the hyperbolic space. Suppose $F_{0}$
  satisfies $\sup_{F_{0}} ( | h |^{2} - \beta ( | H |^{2} ) ) <0$ and $| H
  |^{2} >-25c$. Then the mean curvature flow with initial value $F_{0}$ has a
  unique smooth solution on a finite maximal time interval, and the solution
  $F_{t}$ converges to a round point.
\end{proposition}

Let $\mathring{\beta} ( x ) = \beta ( x ) - \frac{x}{5}$. By direct
computations, for $x>-25c$, we have:
\begin{enumerateroman}
  \item $\max \left\{ \frac{x}{20} +2c,0 \right\} < \mathring{\beta} ( x ) <
  \mathring{\alpha} ( x )$,

  \item $\mathring{\beta} ( x ) \cdot \left( \mathring{\beta} ( x ) +
  \frac{x}{5} -5 c \right) < | H |^{2}   \mathring{\beta}' ( x ) \cdot \left(
  \mathring{\beta} ( x ) + \frac{x}{5} +5 c \right)$,

  \item $2 \sqrt{x}   \mathring{\beta}' ( x ) < \sqrt{\mathring{\beta} ( x
  )}$,

  \item $2x  \mathring{\beta}'' ( x ) + \mathring{\beta}' ( x ) <
  \frac{8}{35}$.
\end{enumerateroman}
With these properties, we can prove Proposition \ref{fived} as the same as the
previous parts of this paper. Notice that $\frac{1}{4} | H |^{2} +2c< \beta (
| H |^{2} )$. Hence Proposition \ref{fived} improves Theorem A for $n=5$.

When $n=4$, let $M$ be a 4-dimensional oriented, simply connected
and complete submanifold immersed in $\mathbb{H}^{4+q} ( c )$.
Suppose that $M$ satisfies $\sup_{M} ( | h |^{2} - \alpha ( 4, | H |
,c ) ) <0$ and $| H |^{2} >-16c$. This pinching condition implies
$M$ is a closed submanifold satisfying $| h |^{2} <4c+ \frac{1}{2} |
H |^{2}$. From Theorem 4.1 of {\cite{MR2550209}}, we get that $M$
has positive isotropic curvature and is diffeomorphic to the
standard 4-sphere.

When $n=3$, let $M$ be a 3-dimensional oriented complete submanifold
immersed in $\mathbb{H}^{3+q} ( c )$. Suppose $M$ satisfies $| h
|^{2} <9c+ \frac{3}{4} | H |^{2} - \frac{1}{4} \sqrt{| H |^{4} +24c
| H |^{2}}$ and $| H |^{2} >-27c$. Then $M$ is diffeomorphic to a
spherical space form. This proposition is the 3-dimensional case of
Theorem 1.2 of {\cite{MR2437072}}.

For the case $c>0$, the following conjecture proposed by
Liu-Xu-Ye-Zhao {\cite{liu2011extension}} is still open up to now.

\begin{conjii}
  Let $M_{0}$ be a complete submanifold immersed in a sphere
  $\mathbb{S}^{n+q} ( 1/ \sqrt{c} )$. Suppose that $\sup_{M_{0}} (
  | h |^{2} - \alpha ( n, | H | ,c ) ) <0$. Then the mean curvature flow with
  initial value $M_{0}$ converges to a round point in finite time, or
  converges to a totally geodesic sphere as $t \rightarrow \infty$. In
  particular, $M_{0}$ is diffeomorphic to the standard $n$-sphere
  $\mathbb{S}^{n}$.
\end{conjii}

\end{document}